\documentclass[preprint,12pt]{elsarticle}

\usepackage[pages=all, color=black, position={current page.south}, placement=bottom, scale=1, opacity=1, vshift=5mm]{background}
\usepackage[margin=1in]{geometry} % full-width

\usepackage{amssymb}
\usepackage{amsmath}
\usepackage{amsmath}
\usepackage{amsthm}
\usepackage{amssymb}
\usepackage{float}
\usepackage{multirow}
\usepackage{tikz}
\usepackage{enumerate}
\usepackage[caption=false]{subfig}
\usepackage[utf8]{inputenc}
\usepackage{hyperref}

\usepackage{oubraces}
\usepackage{stix}
\usepackage{stmaryrd}
\theoremstyle{plain}
\newtheorem{theorem}{Theorem}[section]

\newtheorem{lemma}[theorem]{Lemma}

\theoremstyle{definition}

\newtheorem{example}{Example}
\newtheorem{Remark}[theorem]{Remark}

\newcommand\bx{\boldsymbol{x}}
\newcommand\bbeta{{\boldsymbol{\beta}}}
\newcommand\calP{{\mathcal{P}}}
\newcommand\calF{\mathcal{F}}
\newcommand\frakh{\mathfrak{h}}

\usepackage{graphicx, color}
\graphicspath{{fig/}}

\usepackage{algorithm, algpseudocode} % use algorithm and algorithmicx for typesetting algorithms
\usepackage{mathrsfs} % for \mathscr command

\usepackage{lipsum}
\journal{Journal of Computational and Applied Mathematics}

\begin{document}

\begin{frontmatter}



\title{A nodally bound-preserving finite element method for
	time-dependent convection-diffusion equations}

% \author[label1,label2]{Abdolreza  Amiri, Gabriel R. Barrenechea}
% \affiliation[label1]{organization={},
 \author[label1]{Abdolreza  Amiri} \author[label1]{Gabriel R. Barrenechea}
 \author[label2]{Tristan Pryer}
 \affiliation[label1]{organization={ Department of Mathematics and Statistics, University of Strathclyde},
	             addressline={26 Richmond Street},
	             postcode={G1 1XH},
	             city={Glasgow},
	             country={Scotland}}

 \affiliation[label2]{organization={Department of Mathematical Sciences},
             addressline={University of Bath },
             city={Claverton down, Bath},
             postcode={BA2 7AY},
             country={UK}}

\author{} %% Author name

%\affiliation{organization={},%Department and Organization

%% Abstract
\begin{abstract}
This paper presents a new method to approximate the time-dependent
convection-diffusion equations using conforming finite
element methods, ensuring that the discrete solution respects the
physical bounds imposed by the differential equation.  The method is
built by defining, at each time step, a convex set of admissible
finite element functions (that is, the ones that satisfy the global
bounds at their degrees of freedom) and seeks for a discrete solution
in this admissible set.  A family of $\theta$-schemes is used as time
integrators, and well-posedness of the discrete schemes is proven for
the whole family, but stability and optimal-order error estimates are
proven for the implicit Euler scheme. Nevertheless, our numerical
experiments show that the method also provides stable and
optimally-convergent solutions when the Crank-Nicolson method is used.
\end{abstract}

%% Keywords
\begin{keyword}
	Time-dependent convection-diffusion equation; stabilised finite-element approximation; positivity preservation; variational inequality.



\end{keyword}

\end{frontmatter}
\section{Introduction}
\label{sec:intro}

In many applications the simulation of time-dependent convection-diffusion equations is required. The solution of this equation provides a theory for the transport-controlled reaction rate of two molecules in a shear flow, also it models the processes of the chemical
reaction in a flow field. Such simulations are obtained by a nonlinear system of time dependent convection-diffusion equations for the concentrations of the reactants and the products. An inaccuracy in one of the equations in the system of nonlinear equations affects all the concentrations. Usually in applications the convection field is dominant over the diffusion by several orders of magnitude,
thus making the numerical approximation of such problems very unstable and prone to spurious oscillations. For example,  it is well-known that the standard Galerkin finite element method
for space discretisation is not appropriate for convection-dominated problems.  The preferred way to deal with this fact over the last few decades has been to replace
the standard Galerkin finite element method by a {\sl stabilised} finite element method, that is, a method that adds an extra term to the formulation in such a way that
stability is enhanced.  The first stabilised methods for convection-diffusion problems, e.g., the SUPG method \cite{MR0679322} and the Galerkin Least-Squares method \cite{HUGHES1989173}
were of {\sl residual} type, i.e., a weighted term based on the residual was added to the formulation to add numerical diffusion. In the time dependent context the residual character
of the SUPG method requires  the time derivative of the solution to be included in the stabilisation term, thus creating a somewhat artificial coupling between the time step and
the stabilisation parameter. This coupling has been extensively studied in the literature, see, e.g.,\cite{BURMAN20101114} where stability is analysed for different choices of time discretisation. In fact, if a standard stability analysis is performed, the addition of SUPG stabilisation suggests a loss of stability if the mesh size and
time step are not correctly balanced (although some numerical evidence contradicting this claim can be found, e.g., in \cite{Bochev04}).  Due to this, in later years the
possibilty of using non-residual, symmetric stabilisations for the time-dependent convection-diffusion has been studied more intensively.  Examples of symmetric stabilisations
include the subgrid viscosity method \cite{MR1736900}, the orthogonal subscale method  \cite{MR2068903}, and the continuous interior penalty (CIP) method  \cite{MR2121360}. 
In the work \cite{BF09} a detailed analysis on the stability analysis and advantages of using symmetric stabilisation can be found.

Now, none of the methods mentioned in the previous paragraph can be proven to satisfy the so-called Discrete Maximum Principle (DMP), or even the weaker property of
respecting the bounds satisfied by the continuous problem. In fact, this property has been derived for only a selection of methods,  usually under geometric conditions
on the computational mesh.  For example, in the pioneering work \cite{CR73} it is shown that for the finite element discretisation of a diffusion equation in two space dimensions using piecewise
linear elements a sufficient condition for the satisfaction of the DMP is that the mesh satisfies
the Delaunay condition. This condition was then proven necessary for the matrix to have the right sign pattern in the later work \cite{XZ99}.  The situation is more
dramatic for convection-diffusion equations, where the mesh needs to be acute, and sufficiently fine, for the matrix to have the right sign pattern, and thus the
finite element method to satisfy the DMP.  For convection-diffusion equations, the more straightforward remedy is to add isotropic artificial diffusion to the problem, in such
a way that the dominating matrix in the linear system is the diffusion one.  This strategy does lead to a method that satisfies the DMP under less stringent conditions on
the fineness of the mesh. Unfortunately, the consistency error introduced by adding artificial diffusion is too large, and the results obtained by this method tend to
smear the layers excessively (see, e.g., \cite[Section~5.2]{BJK23} for a discussion and further references).  In the quest of adding diffusion so the DMP is satisfied
and layers are not excessively smeared, the idea of {\sl localising} the diffusion has been proposed as a way to make the problem locally diffusion-dominated in areas near
extrema and layers, thus preventing spurious oscillations and local violations of the discrete maximum principle. This leads
to nonlinear (shock-capturing related) discretisations.  In the last few decades
numerous nonlinear discretisations for the steady-state convection-diffusion equation have been proposed, e.g., 
\cite{MH85,XZ99,BE05,Kuz07,BBK17-NumMath}, and \cite{BJK23} for a review.

In many instances, the numerical solution does not need to be completely free of local spurious oscillations; it only needs to satisfy  global bounds to lead to a stable discretisation.  
One way to achieve a solution that respects this bound is to simply cut the solution at the bound, like it is done in, e.g., the cut-off finite element method \cite{LHV13}. However, 
this process results in a solution that is not an element of the finite element space, and thus analysing its stability and convergence is, in general, a challenging affair.
In the recent works \cite{BGPV23,ABT} an alternative path is followed. It uses the 
following property: There is a correspondence between the imposition of bounds on the numerical solution and searching for the numerical solution on a convex subset of the finite element space of the steady state convection-diffusion equation consisting of the discrete functions satisfying those bounds at their nodal values.
So, in those works a finite element method that seeks {\it directly} for a finite element solution that satisfies the global bounds is proposed, and analysed in different contexts.
In this work we build on those previous works and extend that method to the 
time dependent convection-diffusion equation.  The basic idea can be summarised as follows: for each time step we 
define a set $V_{\mathcal{P}}^{+}$, of {\it admissible} finite element functions as those satisfying the global bound at {\it at their
	degrees of freedom} (nodal values in the case of Lagrangian elements); then, introduce an algebraic projection onto the admissible
set,  denote by $u_h^+$ the projection of $u_h^{}$ onto $V_\calP^+$
and write a finite element problem for the projected object. To eliminate the non-trivial kernel (and so to remove the singularity) generated by this process, a stabilisation term is incorporated into the discretised equations at each time step to remove this kernel (that is, the part of the solution which has been removed).  In our numerical experience we have observed that more stable
results are obtained if we add a linear stabilised term to the problem for $u_h^+$, and thus we present the method including CIP stabilisation.

The remainder of the paper is organised as follows: In Section
\ref{sec:pre} we introduce the notation, the model problem, and all
the preliminary material for the setup of the method.  In
Section~\ref{Sec:FEM} we present the finite element method and show
its well-posedness. The stability and error analysis is carried out in
Section~\ref{Sec:Error}, and in Section \ref{sec:numerics} we test the
performance of the method via different numerical experiments. Finally, some conclusions are drawn in Section \ref{sec:conc}.

\section{General setting and the model problem}
%%%%%%%%%%%%%%%%%%%%%%%%%%%%%%%%%%%%%%%%%%%%%%%%%%%%%%%%%%%%%%%%%%%%%%%%%%%%%%%%%%%%%%%%%%%%%%%%%%%%%%
\label{sec:pre}	
We will adopt standard notations for Sobolev spaces  in line with,
e.g., \cite{EG21-I}. For $D\subseteq\mathbb{R}^{d}$, $d=1,2,3$, we denote by
$\|\cdot\|_{0,p,D} $ the $L^{p}(D)$-norm; when $p=2$ the subscript $p$
will be omitted and we only write $\|\cdot\|_{0,D} $. The inner product in $L^{p}(D)$ is also denoted by $(\cdot,\cdot)_{D}$. In addition, for
$s\geq 0$, $p\in [1,\infty]$, we denote by $\| \cdot \|_{s,p,D}$ ($|
\cdot |_{s,p,D}$) the norm (seminorm) in $W^{s,p}(D)$; when $p=2$, we
define $H^{s}(D)=W^{s,2}(D)$, and again omit the subscript $p$ and only write $\|\cdot \|_{s,D}$
($| \cdot |_{s,D}$).  The following space will also be used repeatedly within the text
\begin{align}
	H^{1}_{0}(D)=\left\{ v \in H^1(D) : v = 0 \; {\rm on} \; \partial D \right\}. \label{space}
\end{align}

For $1\leq p\leq +\infty$, $L^{p}((0,T);W^{s,p}(D))$ is the space defined by
\begin{equation*}
	L^{p}((0,T);W^{s,p}(D))=\left\{u(t,\cdot)\in W^{s,p}(D)~~\text{for almost all $t \in [0,T]$}  : \| u \|_{s,p,D}\in L^{p}(0,T)\right\}\,,
\end{equation*}
and it is a Banach space for the norm
\begin{equation*}
	\begin{split}	
		\|	u \|_{L^{p}((0,T);W^{s,p}(D))}	=	\left\{ \begin{array}{ll} \left(\int_{0}^{T}\| u \|_{s,p,D}^{p} \textrm{d}t\right)^{\frac{1}{p}} \hspace{2.3cm}{\rm if}  \hspace{0.2cm}1\leq p < \infty,\\
			\\
			{\rm ess} \hspace{0.05cm}\sup_{t\in (0,T)}\| u \|_{s,p,D}\hspace{1.9cm}{\rm if}\hspace{0.3cm} p=\infty.
		\end{array} \right.
	\end{split}		
\end{equation*}
%%%%%%%%%%%%%%%%%%%%%%%%%%%%%%%%%%%%%%%%%%%%%%%%%%%%%%%%%%%%    End Remark 1.1
%%%%%%%%%%%%%%%%%%%%%%%%%%%%%%%%%%%%%%%%%%%%%%%%%%
%%%%%%%%%%%%%%%%%%%%%%%%%%%%%%%%%%%%%%%%%%%%%%%%%%%%%%%%%%%%    Sub Section Remark 1.2
%%%%%%%%%%%%%%%%%%%%%%%%%%%%%%%%%%%%%%%%%%%%%%%%%%
\subsection{The model problem} 

Let $\Omega$ be an open bounded Lipschitz domain in $\mathbb{R}^{d}$ ($d=2,3$) with polyhedral boundary $\partial \Omega$, and $T>0$.
For a given $f\in L^{2}((0,T);L^{2}(\Omega))$, we consider the following convection-diffusion problem:
%%%%%%%%%%%%%%%%%%%%%%%%%%%%%%%%%%%%%%%%%%%%%%%%%%%%%%%%%%%%%%%%%%%%%%%%%%%%%%%%%%%%%%%%%%%%%%%%%%%%%%%%%%Convection diffusion equation
%%%%%%%%%%%%%%%%%%%%%%%%%%%%%%%%%%%%%%%%%%%%%%%%%%%%%%%%%%
\begin{equation}
	\begin{split}
		\left\{
		\begin{aligned}
			\partial_{t}u - \varepsilon \,\Delta u + \bbeta \cdot \nabla u + \mu u &= f  &&\text{in } (0,T] \times \Omega, \\
			u(\bx, t) &= 0  &&\text{on } (0,T] \times \partial \Omega, \\
			u(\cdot, 0) &= u^{0} &&\text{in } \Omega,
		\end{aligned}
		\right.\label{CDR}
	\end{split}
\end{equation}
%%%%%%%%%%%%%%%%%%%%%%%%%%%%%%%%%%%%%%%%%%%%%%%%%%%%%%%%%%%%%%%%%%%%%%%%%%%%%%%%%%%%%%%%%%%%%%%%%%
where $\varepsilon\in \mathbb{R}^{+}$, $\bbeta=( \beta_{i})_{i=1}^{d}\in L^{\infty}((0,T);W^{1,\infty}(\Omega))^{d}$,  and $\mu\in \mathbb{R}^{+}_{0}$, respectively, are the diffusion coefficient, the convective field, and the reaction coefficient. 
We will assume that $ {\rm div} \bbeta=0$ in $\Omega\times[0,T]$.

The standard weak formulation of  \eqref{CDR} reads as follows: Find $u\in L^\infty((0,T),H^1_0(\Omega))\cap H^1((0,T),H^{-1}(\Omega))$ such that, 
for almost all $t\in (0,T)$ the following holds
\begin{equation}
	\begin{split}
		\left\{ \begin{array}{ll}  
			(\partial_{t}u,v)_{\Omega} + a(u,v) = (f,v)_{\Omega} \hspace{1cm} \forall v\in H^{1}_{0}(\Omega),\\
			\hspace{1.6cm} u(\cdot,0) = u^{0}.
		\end{array} \right.\label{eq82}
	\end{split}
\end{equation}
Here, the bilinear form $a(\cdot,\cdot)$ is defined by
\begin{equation}
	a(w,v):=\varepsilon\left(\nabla w,\nabla v\right)_{\Omega}+(\bbeta\cdot\nabla w,v)_{\Omega}+\mu( w,v)_{\Omega}\hspace{1cm}\forall v\in 	H^{1}_{0}(\Omega), \hspace{0.4cm} t\in (0,T).\label{eq81}
\end{equation}
In the above definition we have slightly abused the notation, as the convective term $\bbeta$ might depend on $t$, but unless the context requires it, we will always denote
this bilinear form by $a(\cdot,\cdot)$.  Since we have supposed that $\bbeta$ is solenoidal,  for each $t\in (0,T)$ the 
bilinear form  $a(\cdot,\cdot)$ induces the following ``energy'' norm in $	H^{1}_{0}(\Omega)$
\begin{equation*}
	\|v \|_{a}=\sqrt{a(v,v)}\hspace{0.4cm}  t\in [0,T].
\end{equation*}

The well-posedness of  \eqref{eq82} is a well-studied problem. In fact, this is a consequence of Lions' Theorem (see, e.g., \cite[Theorem~10.9]{MR2759829}).
Moreover, as a consequence of the maximum principle for parabolic partial differential equations (see, e.g., \cite[Theorem 12, Section 7.1]{evans2010partial}), the solution of \eqref{eq82}
reaches its extrema on $\big[\Omega \times \{0\} \big]\cup \big[\partial \Omega \times (0,T]\big]$ if $f=0$, and its extrema depend also on the values of $f$ otherwise.  Motivated by this, 
we make the following assumption on $u$, solution of \eqref{eq82}.

\noindent\underline{Assumption (A1):} We will suppose that the weak solution of \eqref{eq82} satisfies 
\begin{equation}
	0\leq u(\bx,t)\leq \kappa(t) \hspace{0.5cm}\text{for almost all}\ \  (\bx,t)\in  \Omega\times[0,T],\label{eq14}
\end{equation}
where $\kappa(t)$ is a known positive function dependent on $t$.

In the above assumption, the lower bound in \eqref{eq14} is not required to be zero; however, we set it to zero for the sake of clarity in the explanation. Additionally, the results presented in this work will remain valid if $\kappa(t)$ is replaced by a positive function $\kappa(\bx,t)$.

%%%%%%%%%%%%%%%%%%%%%%%%%%%%%%%%%%%%%%%%%%%%%%%%%%%
\subsection{Space discretisation}	

As it is usual in the discretisation of parabolic partial differential equations, we first discretise \eqref{eq82} only in space. 
To build a finite element space of $H^1_0(\Omega)$, 
let $\calP$ be a conforming, shape-regular partition of $\Omega$ into simplices (or affine quadrilaterals/hexahedra). 
Over $\calP$, and for $k\geq 1$, we define the finite element space
\begin{align}
	V_\calP^{}:=\{v_{h}\in C^{0}(\overline{\Omega}):v_{h}^{}|_{K}\in \mathfrak{R}(K)\ \ \  \forall K\in \calP \}\cap H^{1}_{0}(\Omega),\label{eq1}
\end{align}
where 
\begin{align}
	\mathfrak{R}(K)= \left\{ \begin{array}{ll} \mathbb{P}_{k}(K),  \hspace{1cm}\text{ if $K$ is a simplex},\\  \mathbb{Q}_{k}(K) , \hspace{1cm}\text{ if $K$ is an affine quadrilateral/hexahedron},\end{array} \right.
\end{align}  
where $\mathbb{P}_{k}(K)$ denotes the polynomials of total degree $k$ on $K$ and $\mathbb{Q}_{k}(K)$ denotes the mapped space of polynomial of degree of at most $k$ in each variable.

For a mesh $\mathcal{P}$, the following notations are used:
\begin{itemize}
	\item let $\{\bx_{1}^{},\bx_{2}^{},\ldots, \bx_{N}^{}\}$ denote the
	set of internal nodes; the usual Lagrangian basis functions
	associated to these nodes, spanning the space $V_{\calP}^{}$, are
	denoted by $\phi_{1}^{},\ldots,\phi_{N}^{}$;
	
	\item let $\calF_{I}^{}$ denote the set of internal facets,
	$\calF_{\partial}^{}$ denote the set of boundary facets of $\calP$; $\calF_{h}^{}=\calF_{I}^{}\cup \calF_{\partial}^{}$ denote the set
	of all facets of $\calP$; for an element $K\in\calP$ the set of its
	facets is denoted by $\calF_{K}^{}$;
	
	%	\item for $K\in \calP$, $F\in \calF_{h}^{}$, and a node $\bx_{i}^{}$, we define the following neighbourhoods:
	%	\begin{align*}
		%		\omega_{K}^{}&=\bigcup\{K'\in \calP: K\cap K'\neq \emptyset\},
		%		\\
		%		\omega_{F}^{}&=\bigcup\{K\in \calP: F\subset K\},
		%		\\
		%		\omega_{i}^{}&=\bigcup\{K\in \calP: \boldsymbol{x}_{i}\in K\};
		%	\end{align*}
	\item for a facet $F\in \mathcal{F}_{I}$, $\llbracket \cdot
	\rrbracket$ denotes the jump of a function across $F$.
\end{itemize}

The diameter of a set \(G \subset \mathbb{R}^d\) is denoted by \(h_G\), and the mesh size is defined as \(h = \max\{h_{K} : K \in \mathcal{P}\}\). We also define the mesh function \(\mathfrak{h}\) as a continuous, element-wise linear function that represents a local average of element diameters, commonly used in finite element analysis \cite{Makridakis:2018}. To construct this, we introduce the set of \textit{vertices} of the mesh, \(\boldsymbol{v}_1, \ldots, \boldsymbol{v}_M\), and define \(\mathfrak{h}\) as the piecewise linear function specified by the nodal values
\begin{equation}\label{mesh-function-definition}
	\frakh(\boldsymbol{v}_i^{})=\dfrac{\sum_{K: \boldsymbol{v}_i^{}\in K}h_K^{}}{\# \{K: \boldsymbol{v}_i^{}\in K\}}\,.
\end{equation} 

In the construction of the method, and its analysis, the following
mass-lumped $L^{2}$-inner product will be of importance: for every
$v_h^{},w_h^{}\in V_\calP^{}$, we
define
\begin{align}
	(v_{h},w_{h})_{h}=\sum_{i=1}^{M}\frakh (\boldsymbol{x}_{i})^{d}v_{h}(\boldsymbol{x}_{i})w_h^{}(\boldsymbol{x}_{i}),
\end{align}  
which induces the norm $\lvert
v_{h}\lvert_{h}:=(v_{h},v_{h})_{h}^{\frac{1}{2}}$ in $V_\calP^{}$. This
norm is, in fact, equivalent to the standard $L^2(\Omega)$-norm.  More
precisely, the following result, whose proof can be found in
\cite[Propositions~28.5, 28.6]{EG21-II}, will be used repeatedly in our analysis
below: There exist $C,c>0$, independent of $h$, such that
\begin{equation}
	c\,\sum_{i: \bx_i^{}\in K}h_K^d v_h^2(\bx_i^{})\leq \|v_h^{}\|_{0,K}^{2}\leq C\,\sum_{i: \bx_i^{}\in K}h_K^d v_h^2(\bx_i^{})
	\qquad\forall\, K\in \calP, ~~ \forall v_h^{}\in V_{\calP},\label{eq5}
\end{equation}
and thus, as a consequence of the shape-regularity  of the mesh, the following holds:
\begin{equation}
	c\,|v_{h}|_h^{2}\leq \|
	v_h^{}\|_{0,\Omega}^{2}\leq C\,|
	v_h^{}|_h^{2},\label{eq55}
\end{equation}
for all $v_h^{}\in V_\calP^{}$. 

Next, we recall that the Lagrange interpolation operator is defined by
(see, e.g., \cite[Chapter~11]{EG21-I})
\begin{align}
	i_{h}&:\mathcal{C}^{0}(\overline{\Omega})\cap H^1_0(\Omega)\longrightarrow V_{\mathcal{P}},\nonumber\\
	v&\longmapsto i_{h}v=\sum_{j=1}^{N}v(x_{j})\phi_{j}. \label{lagrange}
\end{align}

With the above ingredients, we now state some inequalities and
properties that will be useful in what follows:
\begin{list}{ }{ }
	\item \textbf{a) Inverse inequality:}(\cite[Lemma ~12.1]{EG21-I}) For
	all $m,\ell\in \mathbb{N}_0^{},\,0\le m\leq \ell$ and all $p,q \in
	[1,\infty]$, there exists a constant $C$, independent of $h$, such
	that
	\begin{align}
		| v_h^{}|_{\ell,p,K}^{}\leq Ch_K^{m-\ell+d\left(\frac{1}{p}-\frac{1}{q}\right)}\,| v_h^{}|_{m,q,K}^{}\qquad\forall\, v_h^{}\in V_\calP^{}\,.\label{inverse}
	\end{align}
	%%%%%%%%%%%%%%%%%%%%%%%%%%%%%%%%%%%%%%%%%%%%%%%%%%%%%%%%%%%  Trace Inequality 
	%%%%%%%%%%%%%%%%%%%%%%%%%%%%%%%%%%%%%%%%%%%%%%%%%%%
	\item \textbf{b) Discrete Trace inequality:}
	(\cite[Lemma~2.15]{EG21-I}) There exists $C>0$ independent of $h$
	such that, for every $v\in H^{1}(K)$ the following holds
	\begin{equation}
		\|v \|^{2}_{0,\partial K} \leqslant   C \left(h_K^{-1}\|v \|^2_{0,K}+h_K^{}|v |^{2}_{1,K}\right).\label{trace}
	\end{equation}
	%%%%%%%%%%%%%%%%%%%%%%%%%%%%%%%%%%%%%%%%%%%%%%%%%%%%%%%%%%%  Error estimation of the Lagrange interpolant 
	%%%%%%%%%%%%%%%%%%%%%%%%%%%%%%%%%%%%%%%%%%%%%%%%%%%
	\item \textbf{c) Approximation property of the Lagrange interpolant:}
	(\cite[Proposition~1.12]{EG21-I}) Let  $1\leq \ell \leq k$ and
	$i_{h}$ be the Lagrange interpolant. Then, there exists $C>0$,
	independent of $h$,  such that for all $h$ and $v\in H^{\ell+1}(\Omega) \cap H^1_0(\Omega)$ the
	following holds:
	\begin{equation}
		\|v- i_{h}v \|_{0,K}+h_{K} |v-i_{h}v |_{1,K}\leq Ch^{\ell +1}_{K}|v |_{\ell +1,K}\,.\label{lagranges}
	\end{equation} 
\end{list}

The standard Galerkin semi-discretisation of \eqref{eq82}  using the finite element space (\ref{eq1}) reads:
\begin{equation}
	\begin{split}
		\left\{ \begin{array}{ll}
			\text{For almost all $t\in  (0,T)$, find $u_{h}\in V_{\calP}$ such that}\vspace{.1cm}
			\\  
			(\partial_{t}u_{h},v_{h})_{\Omega} + a(u_{h},v_{h}) = (f,v_{h})_{\Omega} \hspace{1cm} \forall v_{h}\in V_\calP,  \vspace{.1cm}\\
			\hspace{2.4cm} u_{h}(\cdot,0) = i_{h}u^{0}.
		\end{array} \right.\label{eq82412}
	\end{split}
\end{equation}

It is a well-known that the standard Galerkin finite element method for \eqref{eq82} in the convective-dominated regime leads to discrete solutions that are
polluted by global spurious oscillations, and thus, in particular, violate Assumption~(A1) (see, e.g. \cite{RST08} for a comprehensive review).
The usual way of enhancing the stability is to add a linear stabilising term aimed at dampening the
oscillations caused by the dominating convection.  Several alternatives are available, with stabilisations based on the addition of
symmetric semi-positive-definite being among the most popular for time-dependent problems. 
In this work we
have chosen to use the continuous interior penalty (CIP) method
originally proposed in \cite{MR2068903} and anlysed in detail for the time-dependent problem in \cite{BF09}. The CIP 
method adds the following stabilising term to the Galerkin scheme (\ref{eq82412}):
%%%%%%%%%%%%%%%%%%%%%%%%%%%%%%%%%%%%%%%%%%%%%%%%%%%%%%%%%%%%%%%%%%%%%%%%%%%%%%%%%%%%%%%%%%%%%%%%%%
\begin{align}
	J(u_h^{},v_h^{})=\gamma \sum_{F\in\calF_I^{}}\int_F\|\bbeta \|_{0,\infty,F}^{} h_{F}^{2}\llbracket  \boldsymbol \nabla u_{h} \rrbracket \cdot\llbracket \boldsymbol \nabla v_{h}  \rrbracket \, \mathrm{d}s\,, \label{eq10}
\end{align}
%%%%%%%%%%%%%%%%%%%%%%%%%%%%%%%%%%%%%%%%%%%%%%%%%%%%%%%%%%%%%%%%%%%%%%%%%%%%%%%%%%%%%%%%%%%%%%%%%%
where
$\gamma\ge 0$ is a non-dimensional constant, and thus it reads as follows:
\begin{equation}
	\begin{split}
		\left\{ \begin{array}{ll}
			\text{For almost all $t\in  (0,T)$, find $u_{h}\in V_{\calP}$ such that} \vspace{.1cm}
			\\  
			(\partial_{t}u_{h},v_{h})_{\Omega} + a_{J}(u_{h},v_{h}) = (f,v_{h})_{\Omega} \hspace{1cm} \forall v_{h}\in V_\calP,  \vspace{.1cm}\\
			\hspace{2.6cm} u_{h}(\cdot,0) = i_{h}u^{0}.
		\end{array} \right.\label{eq824}
	\end{split}
\end{equation}
where
\begin{equation}
	a_{J}(u_{h},v_{h}):=a(u_{h},v_{h})+	J(u_{h},v_{h})
	\hspace{1cm}\forall v_{h}\in V_\calP\,.\label{eq12}
\end{equation}
It is worth mentioning that, even if in this work we have chosen to add CIP stabilisation, the results presented herein remain
valid if we choose any symmetric stabilisation for the convective term,  in particular,  they are valid if we use any of the stabilised methods
analysed in \cite{BF09}.

Although \eqref{eq12} does help remove spurious oscillations, and provides a stable solution, its discrete solution does not 
preserve the physical bounds (\ref{eq14}).  In the next section we introduce the main ingredients to build a finite element method
that  enforces the bound\ (\ref{eq14}) on its solution.

\subsection{The admissible set}

Assumption~(A1) motivates the introduction of the following \textit{admissible set}, that is, the set of finite element
functions that satisfy the bound \eqref{eq14} at their  degrees
of freedom:
\begin{align}
	V_{\mathcal{P}}^{+}:=\{v_{h}\in V_{\mathcal{P}}: v_h(\boldsymbol{x}_{i})\in[0,\kappa(t)]\ \  \text{for all}\ \ i=1,\ldots,N  \}.\label{eq16}
\end{align} 
Every element $v_{h}\in V_\calP$ can be split as the sum
$v_{h}=v_h^{+}+v_h^{-}$, where $v_h^{+}$ and $v_h^{-}$ are given by
\begin{align}
	v_{h}^{+}=\sum_{i=1}^{M}\max\Big\{0,\min\{v_h(\boldsymbol{x}_{i}),\kappa(t)\}\Big\}\,\phi_i\,\quad {\rm for} \ \ t\in (0,T],\label{eq17}
\end{align}
and 
\begin{align}
	v_h^{-}=v_h-v_h^{+}. \label{eq18}
\end{align}
We refer to $v_{h}^{+}$ and $v_{h}^{-}$ as the \textit{constrained}
and \textit{complementary} parts of $v_{h}$, respectively.  Using this
decomposition we define the following algebraic projection
\begin{equation}
	(\cdot)^{+}:V_{\mathcal{P}}\rightarrow V_{\mathcal{P}}^{+}\quad,\quad
	v_{h}\rightarrow v_{h}^{+}\,. \label{posoperator}
\end{equation}  

\begin{Remark}
	Strictly speaking $(\cdot)^{+}$ should be denoted by $(\cdot)^{+,t}$, as $\kappa(t)$ depends on $t$. To lighten the notation, we will simply use $(\cdot)^{+}$ unless it is necessary  to specify the time. 
\end{Remark}

The following result, whose proof is identical to that of \cite[Lemma~3.2]{ABT}, will be useful in the analysis presented below.

\begin{lemma}\label{lem33}
	Let the operator $(\cdot)^{+}$ be defined in \eqref{posoperator}.
	There exists a constant $C>0$, independent of $h$, such that for all $ t\in [0,T]$ the following holds
	\begin{align}
		\|w^{+}_{h}-v^{+}_{h} \|_{0,\Omega}&\leq C \|w_{h}-v_{h} \|_{0,\Omega},\label{Lip1}\\
		\|v_{h}^{+} \|_{0,\Omega}&\leq C\kappa(t),	\label{bound}
	\end{align}
	for all $w_{h},v_{h}\in V_{\mathcal{P}}$.
\end{lemma}

\section{The finite element method}\label{Sec:FEM}

In this section, we propose a $\theta$-scheme time-space discretisation of (\ref{eq824}). Let $N>0$ be a given positive integer. In what follows, we consider a partition of the time interval $[0,T]$ as $t_{0}=0<t_{1}<t_{2}<\cdots<t_{N}=T$ with the time step size $\Delta t_{n}:=t_{n}-t_{n-1}$. To simplify the notation we assume that the time step size is uniform i.e., $\Delta t_{n}=\Delta t=\frac{T}{N}$. In addition, the discrete value $u_{h}^{n}\in V_\calP$ stands for the approximation of $u^{n}=u(t_{n})$ in $V_\calP$ for $0\leq n \leq N$. For $\theta \in \left[\frac{1}{2},1\right]$, we denote
\begin{align*}
	\delta u_{h}^{n}&:=\frac{u_{h}^{n}-u_{h}^{n-1}}{\Delta t}\,,\\
	t_{n-1+\theta}&=\theta t_{n} +(1-\theta)t_{n-1}\quad,\quad
	u_{h}^{n-1+\theta} :=\theta u_{h}^{n} +(1-\theta)u_{h}^{n-1}\,.
\end{align*}

With these notations, the finite element method used in this work reads as follows:
\begin{equation}
	\left\{ \begin{array}{ll}
		\text{For $1\leq n \leq N$, find $u_{h}^{n}\in V_{\calP}$ such that}\vspace{.1cm}
		\\  
		(\delta (u_{h}^{n})^{+},v_{h})_{\Omega} + a_{h}(u_{h}^{n-1+\theta};v_{h}) = (f^{n-1+\theta},v_{h})_{\Omega} \hspace{1cm} \forall v_{h}\in V_\calP, \vspace{.1cm}\\
		\hspace{4.1cm} u_{h}^{0} = i_{h}u^{0}.
	\end{array} \right.\label{eq199}
\end{equation}
Here,  $a_{h}(\cdot;\cdot)$ is defined as 
\begin{equation}
	a_{h}(u_{h}^{n-1+\theta};v_{h}):=\theta a_{J}((u_{h}^{n})^{+},v_{h})+(1-\theta)a_{J}((u_{h}^{n-1})^{+},v_{h})+s((u_{h}^{n})^{-},v_{h})\,,
\end{equation}
where  the stabilisation term $s(\cdot,\cdot)$ is defined as
\begin{align}
	s(v_{h},w_{h})&:=\alpha\sum_{i=1}^{M}\left(\varepsilon\mathfrak{h} (\boldsymbol{x}_{i})^{d-2}+\|\bbeta(\boldsymbol{x},t_{n}) \|_{0,\infty,\omega_{i}}\mathfrak{h} (\boldsymbol{x}_{i})^{d-1}+\left(\frac{1}{\Delta t}+\mu\right)\mathfrak{h} (\boldsymbol{x}_{i})^{d}\right)v_{h}(\boldsymbol{x}_{i})w_{h}(\boldsymbol{x}_{i}).\label{eq202}
\end{align}

Setting $\tilde{\varepsilon}=\Delta t\theta\varepsilon,\tilde{\mathcal{\bbeta}}=\Delta t\theta \bbeta , \tilde{\mu}=(\mu \Delta t\theta+1)$ and $\tilde{J}(\cdot,\cdot)=\Delta t\theta J(\cdot,\cdot)$, 
we can define the following norm at each time step $t^n$:
\begin{align}
	\|v_{h} \|_{h, \theta\Delta t}
	:=
	\left(\tilde{\varepsilon} \|\nabla v_{h}\|^{2}_{0,\Omega}
	+
	\tilde{\mu}\lVert v_{h}\lVert_{0,\Omega}^{2}
	+
	\tilde{J}(v_{h},v_{h})\right)^{\frac{1}{2}}.
	\label{norm}
\end{align}
In addition, the
stabilising form $s(\cdot,\cdot)$ for $0\leq n \leq N$ induces the following norm on $V_\calP$
\begin{align}
	\|v_{h}\|_{s}:=\sqrt{	s(v_{h},v_{h})}\,.\label{eq21}
\end{align}

The following result is a direct consequence of \eqref{eq5} (see  \cite[Lemma~3.1]{ABT}  for the proof of a very similar result). 

%The following result that is a consequence of \eqref{eq5},
%and the proof is identical that of \cite[Lemma~3.1]{ABT} shows that the stabilising bilinear form $s(\cdot,\cdot)$ indeed
%controls $u_h^-$, more specifically it controls the kernel of the
%projection $(\cdot)^+$.
\begin{lemma}\label{Lem:s}
	There exists a constant $C_{\rm equiv}^{}>0$, depending only on the shape regularity of $\calP$, such that
	\begin{align}
		\| v_{h}\|_{h, \theta\Delta t}^{2}\leq \Delta t\frac{C_{\rm equiv}}{\alpha}\|v_{h} \|_{s}^{2} \qquad\forall v_{h}\in V_\calP\,. \label{eq22}
	\end{align}
\end{lemma}

\begin{Remark} At each time-level $0\leq n \leq N$, the finite element method (\ref{eq199}) is a particular case of the finite element method proposed in \cite{ABT}. In fact, at each time step $1\leq n \leq N$, 
	\eqref{eq199} can be written as
	\begin{align}
		((u_{h}^{n})^{+},v_{h})_{\Omega}+	\Delta t\theta a_{J}((u_{h}^{n})^{+},v_{h})+\Delta t s( (u_{h}^{n})^{-} ,v_{h})
		=
		F^{n}(v_{h})\hspace{1cm}\forall v_{h}\in V_\calP, \label{elliptic11}
	\end{align}
	where $a_{J}(\cdot,\cdot)$ is defined in \eqref{eq12},
	and
	\begin{align}
		F^{n}(v_{h})&:=\Delta t (f^{n-1+\theta},v_{h})_{\Omega}-\Delta t(1-\theta)	\varepsilon\big(\nabla (u_{h}^{n-1})^{+},\nabla v_{h}\big)_{\Omega}-
		\Delta t(1-\theta)	(\bbeta\cdot\nabla \left(u_{h}^{n-1})^{+},v_{h}\right)_{\Omega}\nonumber\\&\quad	-
		(\mu \Delta t(1-\theta)-1)\left( (u_{h}^{n-1})^{+},w_{h}\right)_{\Omega}
		-
		\Delta t(1-\theta)	J((u_{h}^{n-1})^{+},v_{h})\hspace{0.7cm}\forall v_{h}\in V_{\mathcal{P}}.\label{RHS}
	\end{align}
	The realisation that at each time step the method \eqref{eq199} is related to the method proposed in \cite{ABT} will be instrumental in the well-posedness result presented in the next section. $\Box$
\end{Remark} 

\subsection{Well-posedness}
In this section, we analyse  the well-posedness of \eqref{eq199}. 
The first step is given by the following
monotonicity result, whose the proof is similar to that of
\cite[Lemma~3.1]{BGPV23}.
%%%%%%%%%%%%%%%% %%%%%%%%%%%%%%%%%%%%%%%%%%%%%%%%%%%%%
%%%%%%%%%%%%%%%%%%%% Begin Lemma 4 %%%%%%%%%%%%%%%%%%%%%%%%%%%%%%%%%%%%%%%%%%%%%%%%%%%%%%%%%%%%%%%%%%%%%%%%%%%%%%%%%%%% 
\begin{lemma}\label{l31}
	The bilinear form $s(\cdot,\cdot)$ defined in \eqref{eq202} for $0\leq n \leq N$ satisfies
	the following inequalities: 
	\begin{align}
		s(v_{h}^- -w_{h}^{-}, v_{h}^{+}-w_h^+)\geq&\, 0\,,\label{eq23}\\
		s( v_{h}^-,w_{h}^+- v_{h}^+)\leq&\, 0\,,\label{eq24}
	\end{align} 
	for every $v_h, w_h\in V_\calP$.
\end{lemma}

The well-posedness of \eqref{eq199} is addressed now. For this, the connection pointed out at the end of the last section is exploited. Namely,
we use an approach very similar to the one used in \cite[Theorem~3.1]{ABT} to prove that, for each $n=1,\ldots,N$ the problem 
\eqref{elliptic11} has a unique solution, which imples that the problem \eqref{eq199} is well-posed.

%%%%%%%%%%%%%%%%%%%% Begin Theorem 5
%%%%%%%%%%%%%%%%%%%% %%%%%%%%%%%%%%%%%%%%%%%%%%%%%%%%%%%%%%%%%%%%%%%%%%%%%%%%%%%%%%%%%%%%%%%%%%%%%%%%%%%%

\begin{theorem}
	\label{Theorem14}
	Let $n=1,\ldots, N$, then,
	\begin{enumerate}[a.]
		\item There exists $u_{h}^{n}\in V_{\mathcal{P}}$ that solves \eqref{elliptic11}.
		\item $(u_{h}^{n})^{+}\in V_\calP^+$ satisfies
		\begin{align}
			((u_{h}^{n})^{+},v_{h}-(u_{h}^{n})^{+})_{\Omega}+	\Delta t\theta a_{J}((u_{h}^{n})^{+},v_{h}-(u_{h}^{n})^{+})	\geq F^n(v_{h}-(u_{h}^{n})^{+})\qquad\forall v_{h}\in V_{\mathcal{P}}^{+}.\label{variational}
		\end{align}
		\item  $(u_{h}^{n})^{-}$ is the unique solution of
		\begin{align}
			\Delta ts((u_{h}^{n})^{-},v_{h})=F^{n}(v_{h})-	((u_{h}^{n})^{+},v_{h})_{\Omega}+	\Delta t\theta a_{J}((u_{h}^{n})^{+},v_{h}) \qquad\forall v_{h}\in V_{\mathcal{P}}.\label{Sunique}
		\end{align} 
		\item   The solution of \eqref{elliptic11} is unique.	
	\end{enumerate}
\end{theorem}
%%%%%%%%%%%%%%%%%%%%%%%%%%%%%%%%%%%%%
%%%%%%%%%%%%%%%%%%%% End Theorem 5 %%%%%%%%%%%%%%%%%%%%%%%%%%%%%%%%%%%%%%%%%%%%%%%%%%
%%%%%%%%%%%%%%%%%%%%%%%%%%%%%%%%%%%%%
%%%%%%%%%%%%%%%%%%%% Begin proof of Theorem 5 %%%%%%%%%%%%%%%%%%%%%%%%%%%%%%%%%%%%%%%%%%%%%%%%%%
\begin{proof} As was mentioned earlier, the proof of this result has many points in common with that of \cite[Theorem~3.1]{ABT}. So, we will skip many of the technical details
	and will just present the main arguments.
	
	a.	We begin by defining the following bilinear form 
	\begin{align*}
		B(v_{h},w_{h})
		:=
		\Delta t\theta	\varepsilon\left(\nabla v_{h} ,\nabla w_{h}\right)_{\Omega}+
		(\mu \Delta t\theta+1)( v_{h} ,w_{h})_{\Omega}
		+
		\Delta t\theta	J( v_{h} ,w_{h})\hspace{0.5cm}\forall v_{h},w_{h}\in V_{\mathcal{P}},
	\end{align*}
	and the mapping 
	\begin{align*}
		T:&V_{\mathcal{P}}\longrightarrow V_{\mathcal{P}},\\
		\hat{u}_{h}^{n}&\longrightarrow u_{h}^{n}=T(\hat{u}_{h}^{n}),\hspace{0.4cm} n=1,\cdots,N,
	\end{align*}
	where $u_{h}^{n}=T(\hat{u}_{h}^{n})$ solves the following equation
	\begin{align}
		B((u_{h}^{n})^{+},v_{h})+\Delta ts((u_{h}^{n})^{-},v_{h})=F^{n}(v_{h}) -\Delta t\theta	(\bbeta\cdot\nabla ((\hat{u}_{h}^{n})^{+} ,v_{h})_{\Omega},\label{eq49}
	\end{align}
	at each time-level $1\leq n \leq N$	and $F^n(\cdot)$ is defined in \eqref{RHS}.
	We observe that $u_{h}^{n}$ solves \eqref{elliptic11} if and only if
	$T(u_{h}^{n})=u_{h}^{n}$. So, the proof will consist on proving that $T$
	satisfies the hypotheses of Brouwer's fixed point Theorem
	\cite[Theorem~10.41]{renardy2006introduction}.
	
	%%%%%%%%%%%%%%%%%%%%%%%%%%%%%%%%%%%%%%%%%%%%%%%% well-definedness <- that's not a word!
	%%%%%%%%%%%%%%%%%%%%%%%%%%%%%%%%%%%%%%%%
	\noindent i) \underline{$T$ is well-defined:} To prove that $T$ is
	well-defined, we see that \eqref{eq49} is a particular example of
	the method proposed in \cite{BGPV23}. So, using
	\cite[Theorem~3.2]{BGPV23}, there exists a unique solution $u_{h}^{n}\in
	V_\calP$ of \eqref{eq49}, and thus $T$ is well-defined.
	%%%%%%%%%%%%%%%%%%%%%%%%%%%%%%%%%%%%%%%%%%%%%%%% continuousness
	%%%%%%%%%%%%%%%%%%%%%%%%%%%%%%%%%%%%%%%%
	
	\noindent ii) \underline{ $T$ is continuous:} Using the monotonicity
	result proven in \cite[Theorem~3.3]{BGPV23} we obtain, for all
	$v_h,w_h\in V_\calP$
	\begin{equation*}
		B(v_{h}^{+}-w_{h}^{+},v_{h}-w_{h})+\Delta ts(v_{h}^{-}-w_{h}^{-},v_{h}-w_{h})\geq 
		C\|v_{h}-w_{h} \|_{h, \theta\Delta t}^{2}.
	\end{equation*}
	Next,  let $\hat{v}_{h}, \hat{w}_{h}\in V_{\mathcal{P}}$
	and let $v_{h}=T(\hat{v}_h)$ and $w_{h}=T(\hat{w}_h)$.  Then,  a lengthy calculation using the last result, \eqref{eq49}, integration by parts,  H\"older's inequality, 
	Lemma~\ref{lem33}, and \eqref{norm}, leads to the following Lipschitz continuity of the operator $T$:
	%	
	%	using integrating by
	%	parts, Lemma~\ref{lem33}, and \eqref{norm}, we get
	%	\begin{align*}
		%		C\|v_{h}-w_{h} \|_{h, \theta\Delta t}^{2}&\leq	B(v_{h}^{+}-w_{h}^{+},v_{h}-w_{h})+s(v_{h}^{-}-w_{h}^{-},v_{h}-w_{h})\\                                      
		%		&= -\theta\Delta t(\bbeta\cdot \nabla (\hat{v}_{h}^{+}-\hat{w}_{h}^{+}),v_{h}-w_{h})_{\Omega}\\
		%		& =\theta\Delta t(\hat{v}_{h}^{+}-\hat{w}_{h}^{+},\bbeta\cdot \nabla(v_{h}-w_{h}) )_{\Omega}\\
		%		&\leq C\theta\Delta t\| \bbeta\|_{0,\infty,\Omega}\|\hat{v}_{h}-\hat{w}_{h} \|_{0,\Omega}|v_{h}-w_{h}  |_{1,\Omega}\\&\leq 
		%		C\theta\Delta t\tilde{\varepsilon}^{-\frac{1}{2}}\|\bbeta\|_{0,\infty,\Omega}\|\hat{v}_{h}-\hat{w}_{h} \|_{0,\Omega}\tilde{\varepsilon}^{\frac{1}{2}}| (v_{h}-w_{h} ) |_{1,\Omega}
		%		\\&\leq  C\theta\Delta t\tilde{\varepsilon}^{-\frac{1}{2}}\|\bbeta\|_{0,\infty,\Omega} \|\hat{v}_{h}-\hat{w}_{h}\|_{0,\Omega}\|v_{h}-w_{h} \|_{h, \theta\Delta t}.
		%	\end{align*}
	%	Therefore
	\begin{align*}
		\|T(\hat{v}_{h})-T(\hat{w}_{h}) \|_{h, \theta\Delta t}\leq C \theta\Delta\, t\,\tilde{\varepsilon}^{-\frac{1}{2}}\|\bbeta\|_{0,\infty,\Omega}\|\hat{v}_{h}-\hat{w}_{h}\|_{0,\Omega}\,.
	\end{align*}
	%and $T$ is Lipschitz continuous.
	
	%%%%%%%%%%%%%%%%%%%%%%%%%%%%%%%%%%%%%%%%%%%%%%%% Existance of a compact ball
	%%%%%%%%%%%%%%%%%%%%%%%%%%%%%%%%%%%%%%%%
	\noindent iii) \underline{There exists $R>0$, such that
		$T(B(0,R))\subseteq B(0,R)$:} Let $\hat{z}_{h}\in V_\calP$ be
	arbitrary and $z_{h}=T(\hat{z}_{h})$. By using $v_{h}=z_{h}^{+}$ in
	\eqref{eq49},  we get
	\begin{align}
		B(z_{h}^{+},z_{h}^{+})&+\underbrace{\Delta ts(z_{h}^{-},z_{h}^{+})}_{\ge 0} =F(z_{h}^{+}) - \theta\Delta t(\bbeta \cdot \nabla \hat{z}_{h}^{+},z_{h}^{+})_{\Omega}
		\le M\|z_{h}^{+} \|_{h, \theta\Delta t}.\label{324t}
	\end{align}
	In fact, using the Cauchy-Schwarz and Young inequalities,  the fact that $\tilde{\mu}= 1+\mu \theta \Delta t\neq 0$ and  $\theta\geq \frac{1}{2}$, we obtain
	\begin{align}
		|F^n(z_{h}^{+})|&\leq C\left(\Delta t\tilde{\mu}^{-\frac{1}{2}}\parallel f^{n-1+\theta}\parallel_{0,\Omega}+  \Delta t\theta	\|\bbeta\|_{0,\infty,\Omega}\tilde{\varepsilon}^{-\frac{1}{2}}\tilde{\mu}^{-\frac{1}{2}}\|(u_{h}^{n-1})^{+} \|_{h, \theta\Delta t}\right.\nonumber\\&\quad\left.+\|(u_{h}^{n-1})^{+} \|_{h, \theta\Delta t}\right)	\|z_{h}^{+} \|_{h, \theta\Delta t}.\nonumber
	\end{align}	
	Moreover, integration by parts and the H\"older inequality yield
	\begin{equation*}
		\theta\Delta t(\bbeta \cdot \nabla \hat{z}_{h}^{+},z_{h}^{+})_{\Omega}
		%\Delta t\theta	\|\bbeta\|_{0,\infty,\Omega}\tilde{\varepsilon}^{-\frac{1}{2}}\parallel \hat{z}_{h}^{+}\parallel_{0,\Omega}\tilde{\varepsilon}^{\frac{1}{2}}\lvert z_{h}^{+}\rvert_{1,\Omega}\\&
		\leq\theta\Delta t \tilde{\varepsilon}^{-\frac{1}{2}}\|\hat{z}_{h}^{+} \|_{0,\Omega}\| \bbeta\|_{0,\infty,\Omega}\parallel z_{h}^{+} \parallel_{h, \theta\Delta t}.
	\end{equation*}
	So,  setting
	\begin{align}
		M&:=C\left(\Delta t\tilde{\mu}^{-\frac{1}{2}}\parallel f^{n-1+\theta}\parallel_{0,\Omega}+  \Delta t\theta	\|\bbeta\|_{0,\infty,\Omega}\tilde{\varepsilon}^{-\frac{1}{2}}\tilde{\mu}^{-\frac{1}{2}}\|(u_{h}^{n-1})^{+} \|_{h, \theta\Delta t}\nonumber\right.\\& \quad\left.+\|(u_{h}^{n-1})^{+} \|_{h, \theta\Delta t}+\theta\Delta t \tilde{\varepsilon}^{-\frac{1}{2}}\kappa(t^n)\| \bbeta\|_{0,\infty,\Omega}\right), \label{boundM}
	\end{align}
	and using the ellipticity of $B(\cdot,\cdot)$, \eqref{bound}, and \eqref{324t} we arrive at the following bound for $z_h^+$:
	\begin{align*}
		\|z_{h}^{+} \|_{h, \theta\Delta t}
		\leq M\,.
	\end{align*}

	Next, we take $v_{h}=z_{h}^{-}$ in \eqref{eq49}.  Using analogous arguments, that is, integration by parts, H\"older's inequality,  \eqref{bound}, and Lemma~\ref{Lem:s}
	we derive the following bound for $z_h^-$:
	\begin{align*}
		\|z_{h}^{-} \|_{h,\theta\Delta t}
		\leq C_{2}(f^{n-1+\theta},u_{h}^{n-1},\tilde{\varepsilon},\tilde{\mu},\bbeta,\kappa(t^{n}),h,\Delta t),
	\end{align*}
	where $C_{2}(f^{n-1+\theta},u_{h}^{n-1},\tilde{\varepsilon},\tilde{\mu},\bbeta,\kappa(t^n),h,\Delta t)$ is independent of $z_h$.  Hence, $z_h=T(\hat{z}_h)$ satisfies the
	following (uniform) bound
	\begin{align*}
		\|z_{h} \|_{h,\theta\Delta t}&\leq \|z_{h}^{-} \|_{h,\theta\Delta t} +\|z_{h}^{+} \|_{h,\theta\Delta t} \leq   M+C_{2}(f^{n-1+\theta},u_{h}^{n-1},\tilde{\varepsilon},\tilde{\mu},\bbeta,\kappa(t^n),h,\Delta t)=:R.
	\end{align*}
	Therefore, $z_{h}=T(\hat{z}_{h})\in B(0,R)$, for every $\hat{z}_{h}\in
	V_{\mathcal{P}}$, which shows that $T(B(0,R))\subseteq B(0,R)$.
	Hence, using Brouwer's fixed point theorem, there exists one $u_{h}^{n}\in
	V_{\mathcal{P}}$ such that $T(u_{h}^{n})=u_{h}^{n}$. In other words, problem
	\eqref{elliptic11} has at least one solution.
	
	The proofs of (b) and (c) are identical to those of \cite[Lemma~3.3]{ABT}.
	Finally,  the proof of (d) is identical to that of  \cite[Corollary~3.1]{ABT}.
\end{proof}
	\begin{Remark} 
		It is worth mentioning that, due to the equivalence between (\ref{eq199}) and the variational inequality (\ref{variational}), and the well-posedness of the latter, $(u^{n}_{h})^{+}$ is independent of the choice of the stabilisation, as long as it satisfies (\ref{eq22}) and (\ref{eq23}). In particular, all the mentioned methods proven in this work are remained valid if $s(\cdot,\cdot)$
		defined in (\ref{eq202}) is replaced by 
		\begin{align*}
				s(v_{h},w_{h}) := \alpha \sum_{i=1}^{M} \left( \varepsilon \mathfrak{h} (\boldsymbol{x}_{i})^{d-2} 
				+ \|\bbeta(\boldsymbol{x},t_{n}) \|_{0,\infty,\omega_{i}} \mathfrak{h} (\boldsymbol{x}_{i})^{d-1} 
				+ \mu \mathfrak{h} (\boldsymbol{x}_{i})^{d} \right) v_{h}(\boldsymbol{x}_{i}) w_{h}(\boldsymbol{x}_{i}).
		\end{align*}
		In this case, the solution $(u^{n}_{h})^{+}$ remains unchanged, as it still satisfies \eqref{variational}, meaning the overall analysis remains the same.
		
		The inclusion of the factor $\frac{1}{\Delta t}$ in the time derivative was primarily motivated by the performance of the nonlinear solver. Without this factor, the nonlinear solver exhibited significantly slower convergence.
\end{Remark}
%%%%%%%%%%%%%%%%%%%%%%%%%%%%%%%%%%%%%
%%%%%%%%%%%%%%%%%%%% 
%//
%%%%%%% Section 2.2: Error estimation
%%%%%%%%%%%%%%%%%%%%%
\section{Stability and Error analysis}\label{Sec:Error}

This section is devoted to proving a stability result and derive optimal error estimates for the method \eqref{eq199} in the particular case $\theta=1$, that is, in the case the time discretisation
is carried out using the implicit Euler method.  An important tool that will be used throughout is the following discrete Gronwall Lemma, proved originally in
\cite[Lemma~5.1]{heywood1990finite}.

\begin{lemma}\label{13}
	Let $k$, $B$, $a_{n}$, $b_{n}$, $c_{n}$, $\gamma_{n}$, $n=0,\ldots,m$, be non-negative numbers such that
	\begin{align*}
		a_{n}+k\sum_{n=0}^{m}b_{n}\leq k\sum_{n=0}^{m}\gamma_{n}a_{n}+k\sum_{n=0}^{m}c_{n}+B\hspace{0.5cm}{\rm for}\hspace{0.3cm}m\geq 0.
	\end{align*} 
	Suppose $k\gamma_{n}\leq 1$ for every $j$, and set $\sigma_{n}=(1-k\gamma_{n})^{-1}$. Then
	\begin{align}
		a_{m}+k\sum_{n=0}^{m}b_{n}\leq \exp\left(k\sum_{n=0}^{m}\sigma_{n}\gamma_{n}\right)\left(k\sum_{n=0}^{m}c_{n}+B\right)\hspace{0.5cm}{\rm for}\hspace{0.3cm}m\geq 0.\label{inequality11}
	\end{align} 
	
\end{lemma}

For the case $\theta=1$, the method \eqref{eq199} reads: 
\begin{equation}
	\begin{split}
		\left\{ \begin{array}{ll}
			\text{For $1\leq n \leq N$, find $u_{h}\in V_\calP$ such that} \vspace{.1cm} \\
			(\delta (u_{h}^{n})^{+},v_{h})_{\Omega} + a_{h}(u_{h}^{n};v_{h}) = (f^{n},v_{h})_{\Omega} \hspace{1cm} \forall v_{h}\in V_\calP,  \vspace{.1cm} \\
			\hspace{3.45cm} u^{0}_{h} = i_{h}u^{0}.
		\end{array} \right.\label{Euler}
	\end{split}
\end{equation}

To prove the stability, we use the test function $v_{h}=(u^{n}_{h})^{+}$ in \eqref{Euler}, and obtain
\begin{align*}
	(\delta (u_{h}^{n})^{+},(u^{n}_{h})^{+})_{\Omega}&+ \varepsilon(\nabla (u^{n}_{h})^{+},\nabla (u^{n}_{h})^{+})_{\Omega}+(\bbeta\cdot\nabla (u^{n}_{h})^{+},(u^{n}_{h})^{+})_{\Omega}\\&+ \mu((u^{n}_{h})^{+}, (u^{n}_{h})^{+})_{\Omega}+J((u^{n}_{h})^{+},(u^{n}_{h})^{+})+s((u^{n}_{h})^{-},(u^{n}_{h})^{+})=
	(f^{n},(u^{n}_{h})^{+})_{\Omega},\nonumber\hspace{2cm}
\end{align*}
or, equivalently, using that $\left(\bbeta\cdot\nabla (u^{n}_{h})^{+},(u^{n}_{h})^{+}\right)_{\Omega}=0$, 
\begin{align*}
	((u^{n}_{h})^{+}-(u^{n-1}_{h})^{+},(u^{n}_{h})^{+})_{\Omega}&+\Delta t \Big\{\varepsilon\, |(u^{n}_{h})^{+}|^2_{1,\Omega}
	%+(\bbeta\cdot\nabla (u^{n}_{h})^{+},(u^{n}_{h})^{+})_{\Omega}\\&
	+ \mu\,\|(u^{n}_{h})^{+}\|^2_{0,\Omega}+J((u^{n}_{h})^{+},(u^{n}_{h})^{+})\Big\}\\&+\Delta ts((u^{n}_{h})^{-},(u^{n}_{h})^{+})=
	\Delta t(f^{n},(u^{n}_{h})^{+})_{\Omega}.\nonumber\hspace{2cm}
\end{align*}
The relation $2p(p-q)=p^{2}+(p-q)^{2}-q^{2}$, the Cauchy-Schwarz inequality, and the fact that  $s((u^{n}_{h})^{-},(u^{n}_{h})^{+})\geq 0$ (by Lemma \ref{l31})  lead to
\begin{align}
	\parallel(u^{n}_{h})^{+}\parallel_{0,\Omega}^{2}&- \parallel( u ^{n-1}_{h})^{+}\parallel_{0,\Omega}^{2}+\parallel(u^{n}_{h})^{+}-( u ^{n-1}_{h})^{+}\parallel_{0,\Omega}^{2}+2\Delta t \left\{\varepsilon \lvert(u^{n}_{h})^{+}\rvert_{1,\Omega}^{2}\right.\nonumber\\&\left.+ \mu\parallel(u^{n}_{h})^{+}\parallel_{0,\Omega}^{2}+J((u^{n}_{h})^{+},(u^{n}_{h})^{+})\right\}\leq
	2\Delta t\parallel f^{n}\parallel_{0,\Omega}\parallel (u^{n}_{h})^{+}\parallel_{0,\Omega}. \label{eq1212}\hspace{2cm}
\end{align}
Using Young's inequality for the right hand side and then summing through $n$, $n=0,\ldots,m$, we get that
\begin{align*}
	\parallel(u^{m}_{h})^{+}\parallel_{0,\Omega}^{2}&+	\sum_{n=0}^{m}\parallel(u^{n}_{h})^{+}-( u ^{n-1}_{h})^{+}\parallel_{0,\Omega}^{2}+2\sum_{n=0}^{m}\Delta t \left\{\varepsilon \lvert(u^{n}_{h})^{+}\rvert_{1,\Omega}^{2} +\mu\parallel(u^{n}_{h})^{+}\parallel_{0,\Omega}^{2}\right.\nonumber\\&\left. +J((u^{n}_{h})^{+},(u^{n}_{h})^{+})\right\}\leq \parallel u^{0}_{h}\parallel_{0,\Omega}^{2}
	+\sum_{n=0}^{m}\Delta t\left(T\parallel f^{n}\parallel_{0,\Omega}^{2}+\frac{1}{T}\parallel (u^{n}_{h})^{+}\parallel_{0,\Omega}^{2}\right). 
\end{align*}
If we set $a_{n}=\parallel(u^{n}_{h})^{+}\parallel_{0,\Omega}^{2}$, $B=\parallel u^{0}_{h}\parallel_{0,\Omega}^{2}$, $k=1$, $\gamma_{n}=\frac{\Delta t}{T}$ and $\sigma_{n}=\left(1-\frac{\Delta t}{T}\right)^{-1}$, $	c_{n}=\Delta tT\parallel  f^{n}\parallel_{0,\Omega}^{2}$ and
\begin{align*}
	b_{n}=&2\Delta t \left(\varepsilon \lvert(u^{n}_{h})^{+}\rvert_{1,\Omega}^{2}+ \mu\parallel(u^{n}_{h})^{+}\parallel_{0,\Omega}^{2}+J((u^{n}_{h})^{+},(u^{n}_{h})^{+}) \right) ,
\end{align*}
then using the Gr\"{o}nwall's inequality Lemma \ref{13}, we get 
\begin{align*}
	\parallel(u^{m}_{h})^{+}\parallel_{0,\Omega}^{2}&+2\Delta t\sum_{n=0}^{m} \left(\varepsilon \lvert(u^{n}_{h})^{+}\rvert_{1,\Omega}^{2} + \mu\parallel(u^{n}_{h})^{+}\parallel_{0,\Omega}^{2}+J((u^{n}_{h})^{+},(u^{n}_{h})^{+})\right)\\&\leq 
	\exp{\left(\sum_{n=0}^{m}\frac{\Delta t}{T}\left(1-\frac{\Delta t}{T}\right)^{-1}\right)}\left(\parallel u^{0}_{h}\parallel_{0,\Omega}^{2}+\Delta tT\sum_{n=0}^{m}\parallel  f^{n}\parallel_{0,\Omega}^{2}\right)
	.\nonumber\hspace{2cm}
\end{align*}
In this way we have proved the following stability result for the scheme \eqref{Euler}.
\begin{lemma} Let $u^{n}_{h}\in  V_\calP$, for $n=1,\ldots, N$ solve \eqref{Euler}. Then the following stability estimates holds true: 
	\begin{align*}
		\max\limits_{1\leq m\leq N}	\parallel(u^{m}_{h})^{+}\parallel_{0,\Omega}^{2}&+2\Delta t\sum_{n=0}^{N} \left(\varepsilon \lvert(u^{n}_{h})^{+}\rvert_{1,\Omega}^{2} + \mu\parallel(u^{n}_{h})^{+}\parallel_{0,\Omega}^{2}+J((u^{n}_{h})^{+},(u^{n}_{h})^{+})\right)\\&\leq 
		e^{2}\left(\parallel u^{0}_{h}\parallel_{0,\Omega}^{2}+\Delta tT\sum_{n=0}^{N}\parallel  f^{n}\parallel_{0,\Omega}^{2}\right).
	\end{align*} 
\end{lemma}
\begin{Remark}
	In the case particular case $f=0$, then \eqref{eq1212} implies that
	\begin{align*}
		\parallel(u^{m}_{h})^{+}\parallel_{0,\Omega}\leq \parallel u^{0}_{h}\parallel_{0,\Omega}\,,
	\end{align*} 
	and then \eqref{Euler} is strongly stability preserving. $\Box$
\end{Remark}

The next result states optimal order error estimates for the method \eqref{Euler}.

\begin{theorem}\label{Theorem11}
	Let $u^{0}\in H^{k+1}(\Omega)$, $u\in  L^{\infty}((0,T);H^{k+1}(\Omega))\cap H^{1}((0,T);H^{k+1}(\Omega))\cap H^{2}((0,T);L^{2}(\Omega))$ be the solution of (\ref{CDR}), $u(\cdot,t)\in H^{1}_{0}(\Omega)$ for almost all $t\in [0,T]$, and $u_{h}^{n}\in V_{\calP}$ be the solution of (\ref{Euler}) at the time step $n$. Then, defining $E^{m}=(u^{m}_{h})^{+}-u^{m}$, there exists a constant $C>0$,
	independent of $h, \Delta t$ and any physical parameter, such that
	\begin{align}
		\max\limits_{1\leq m\leq N}&\parallel E^{m}\parallel_{0,\Omega}^{2} + \Delta t\sum_{n=1}^{N}\bigg(\varepsilon\lvert  E^{n}_{h}\rvert_{1,\Omega}^{2} + \mu\parallel E^{n}_{h}\parallel_{0,\Omega}^{2} + J(E^{n}_{h},E^{n}_{h})\bigg)\nonumber\\
		& \leq Ce^{2}\Bigg[h^{2k}\Bigg\{\Delta t \sum_{n=0}^{N} \Bigg(\varepsilon  +  T\|\bbeta\|_{0,\infty,\Omega}^{2}  +h\gamma \|\bbeta\|_{0,\infty,\Omega}+ h^{2}T\mu^{2}\Bigg)|u^{n} |_{k+1,\Omega}^{2}\nonumber\\
		& \quad + h^{2}\max\limits_{1\leq m\leq N}|u^{m} |_{k+1,\Omega}^{2} + Th^{2}\int_{0}^{T}|\partial_{t} u(t)|_{k+1,\Omega}^{2}{\rm d}t\Bigg\}  + \Delta t^2 T\int_{0}^{T}\parallel \partial_{tt} u(t)\parallel_{0,\Omega}^{2}{\rm d}t\Bigg].\label{error11}
	\end{align}
\end{theorem}

\begin{proof}
	For $n=1,\ldots,N$ we decompose $E^{n}=(u^{n}_{h})^{+}-u^{n}$ as 
	\begin{align}
		E^{n}=(u^{n}_{h})^{+}-u^{n}=\left((u^{n}_{h})^{+}-i_{h}u^{n}\right)+\left(i_{h}u^{n}-u^{n}\right)=:E^{n}_{h}+\eta^{n}_{h},\label{error1}
	\end{align} 
	where we recall that $i_{h}u^{n}$ is the Lagrange interpolant \eqref{lagrange} of $u^{n}$.
	Subtracting \eqref{eq81} from the method (\ref{Euler}) we arrive at the following error equation
	\begin{align}
		\left(\delta (u_{h}^{n})^{+}-\partial_{t}u^{n},v_{h}\right)_{\Omega}&+\varepsilon\left(\nabla ((u^{n}_{h})^{+}-u^{n}),\nabla v_{h}\right)_{\Omega}+\left(\bbeta\cdot\nabla ( (u^{n}_{h})^{+}-u^{n}),v_{h}\right)_{\Omega}	\nonumber\\&+\mu\left( (u^{n}_{h})^{+}-u^{n},v_{h}\right)_{\Omega}+J((u^{n}_{h})^{+},v_{h})+	s((u^{n}_{h})^{-},v_{h})=0.\label{orthoganal}
	\end{align} 
	Rearranging and using \eqref{error1},we get  
	\begin{align}
		(\delta E^{n}_{h},v_{h})_{\Omega}&+\varepsilon\left(\nabla E^{n}_{h},\nabla v_{h}\right)_{\Omega}+(\bbeta\cdot \nabla E^{n}_{h}, v_{h})_{\Omega}+\mu( E^{n}_{h}, v_{h})_{\Omega}+J((u^{n}_{h})^{+},v_{h})+	s((u^{n}_{h})^{-},v_{h})\nonumber\\&=-(\delta (i_{h}u^{n})-\partial_{t} u^{n},v_{h})_{\Omega}-\varepsilon( \nabla\eta^{n}_{h},\nabla v_{h})_{\Omega}-(\bbeta\cdot \nabla \eta^{n}_{h}, v_{h})_{\Omega}-\mu( \eta^{n}_{h}, v_{h})_{\Omega}\label{Galerkin12}.
	\end{align} 
	Since $u^{n}\in H^{2}(\Omega)$, then $J(u^{n}, v)=0$, and so we deduce that
	\begin{align*}
		J((u^{n}_{h})^{+}, v)=J((u^{n}_{h})^{+}-u^{n}, v)=	J(E_{h}^{n}, v)+J(\eta^{n}_{h}, v).
	\end{align*}
	Using the test function $v_{h}=E^{n}_{h}$ in (\ref{Galerkin12}) and $(\bbeta\cdot \nabla E^{n}_{h}, E^{n}_{h})_{\Omega}=0$, we get
	\begin{align}
		&(\delta E^{n}_{h},E^{n}_{h})_{\Omega}+\varepsilon\,|E^{n}_{h}|^2_{1,\Omega}+\mu\,\|E^{n}_{h}\|^2_{0,\Omega}+J(E^{n}_{h}, E^{n}_{h})+	s((u^{n}_{h})^{-},E^{n}_{h})\nonumber\\&
		=-(\delta (i_{h}u^{n})-\partial_{t} u^{n},E^{n}_{h})_\Omega^{}-\varepsilon\left(\nabla \eta^{n}_{h},\nabla E^{n}_{h}\right)_{\Omega}-(\bbeta\cdot \nabla \eta^{n}_{h},E^{n}_{h})_{\Omega}-\mu( \eta^{n}_{h},E^{n}_{h})_{\Omega}-J(\eta^{n}_{h}, E^{n}_{h}).\label{Galerkin122}
	\end{align} 
	
	Since $(i_{h}u^{n})^{-}=0$,  the monotonicity inequality (\ref{eq24}) yields
	\begin{align}
		s(( u^{n}_{h})^{-},E^{n}_{h})&=	s(( u^{n}_{h})^{-},(u^{n}_{h})^{+}-i_{h}u^{n})=s(( u^{n}_{h})^{-}-(i_{h}u^{n})^{-},(u^{n}_{h})^{+}-(i_{h}u^{n})^{+})\geq 0.\label{monoton1}
	\end{align} 
	So, using the relation $2p(p-q)=p^{2}+(p-q)^{2}-q^{2}$ for the first term of (\ref{Galerkin122}), applying the inequality (\ref{monoton1}), next using the Cauchy–Schwarz inequality and then Young's inequality for 
	the terms on the right hand side, we get
	\begin{align*}
		\parallel E^{n}_{h}\parallel_{0,\Omega}^{2}&-\parallel E^{n-1}_{h}\parallel_{0,\Omega}^{2}+\parallel E^{n}_{h}-E^{n-1}_{h}\parallel_{0,\Omega}^{2} +2\Delta t\left(\varepsilon\lvert  E^{n}_{h}\rvert_{1,\Omega}^{2}+\mu\parallel E^{n}_{h}\parallel_{0,\Omega}^{2}+J(E^{n}_{h}, E^{n}_{h})\right)\\&\leq 2\Delta t \Big(\parallel \delta (i_{h}u^{n})-\partial_{t} u^{n}\parallel_{0,\Omega}\parallel E^{n}_{h}\parallel_{0,\Omega}+\varepsilon\lvert  \eta^{n}_{h}\rvert_{1,\Omega}\lvert  E^{n}_{h}\rvert_{1,\Omega}+\parallel\bbeta\cdot \nabla \eta^{n}_{h}\parallel_{0,\Omega}\parallel E^{n}_{h}\parallel_{0,\Omega}\\&\quad+\mu\parallel \eta^{n}_{h}\parallel_{0,\Omega}\parallel E^{n}_{h}\parallel_{0,\Omega}+J(\eta^{n}_{h}, \eta^{n}_{h})^\frac{1}{2}J(E^{n}_{h}, E^{n}_{h})^\frac{1}{2}\Big) \\ &\leq \Delta t \Big(\varepsilon\lvert  \eta^{n}_{h}\rvert_{1,\Omega}^{2}+\varepsilon\lvert  E^{n}_{h}\rvert_{1,\Omega}^{2}+ T\left(\parallel \delta (i_{h}u^{n})-\partial_{t} u^{n}\parallel_{0,\Omega}+\parallel\bbeta\cdot \nabla \eta^{n}_{h}\parallel_{0,\Omega}+\mu\parallel \eta^{n}_{h}\parallel_{0,\Omega}\right)^{2}\\&\quad+ \frac{1}{T}\parallel E^{n}_{h}\parallel_{0,\Omega}^{2}
		+J(E^{n}_{h}, E^{n}_{h})+J(\eta^{n}_{h}, \eta^{n}_{h})\Big).
	\end{align*} 
	Rearranging terms on the both sides of the inequality yields
	\begin{align*}
		\parallel E^{n}_{h}\parallel_{0,\Omega}^{2}&-\parallel E^{n-1}_{h}\parallel_{0,\Omega}^{2}+	\parallel E^{n}_{h}-E^{n-1}_{h}\parallel_{0,\Omega}^{2} +\Delta t\left(\varepsilon\lvert  E^{n}_{h}\rvert_{1,\Omega}^{2}+2\mu\parallel E^{n}_{h}\parallel_{0,\Omega}^{2}+J(E^{n}_{h}, E^{n}_{h})\right)\\&\leq \Delta t \left(\varepsilon\lvert  \eta^{n}_{h}\rvert_{1,\Omega}^{2}+  T\left(\parallel \delta (i_{h}u^{n})-\partial_{t} u^{n}\parallel_{0,\Omega}+\parallel\bbeta\parallel_{0,\infty,\Omega}\lvert  \eta^{n}_{h}\rvert_{1,\Omega}+\mu\parallel \eta^{n}_{h}\parallel_{0,\Omega}\right)^{2}\right.\\&\left.\quad+\frac{1}{T}\parallel E^{n}_{h}\parallel_{0,\Omega}^{2}+J(\eta^{n}_{h}, \eta^{n}_{h})\right)\\&\leq C\Delta t\left(\left(\varepsilon+T\|\bbeta\|_{0,\infty,\Omega}^{2}\right)\lvert  \eta^{n}_{h}\rvert_{1,\Omega}^{2} +  T\parallel \delta(i_{h}u^{n})-\partial_{t} u^{n}\parallel_{0,\Omega}^{2}+T\mu^{2}\parallel\eta^{n}_{h} \parallel_{0,\Omega}^{2} \right.\\&\quad\left.+\frac{1}{T}\parallel E^{n}_{h}\parallel_{0,\Omega}^2+J(\eta^{n}_{h}, \eta^{n}_{h})\right).
	\end{align*} 
	Summing for $n=0$ to $m$  and using $E^{0}_{h}=0$ leads to
	\begin{align*}
		\parallel E^{m}_{h}\parallel_{0,\Omega}^{2}&+\Delta t\sum_{n=1}^{m}\left(\varepsilon\lvert  E^{n}_{h}\rvert_{1,\Omega}^{2}+2\mu\parallel E^{n}_{h}\parallel_{0,\Omega}^{2}+J(E^{n}_{h}, E^{n}_{h})\right)\\
		&\leq C\,\Delta t\sum_{n=0}^{m}\left\{\left(\varepsilon+T\|\bbeta\|_{0,\infty,\Omega}^{2}\right)\lvert  \eta^{n}_{h}\rvert_{1,\Omega}^{2} +T\parallel \delta (i_{h}u^{n})-\partial_{t} u^{n}\parallel_{0,\Omega}^{2} \right.\\&\left.\quad +T\mu^{2}\parallel\eta^{n}_{h} \parallel_{0,\Omega}^{2} +\frac{1}{T}\parallel E^{n}_{h}\parallel_{0,\Omega}^2+J(\eta^{n}_{h}, \eta^{n}_{h})\right\}.
	\end{align*} 
	
	We are now ready to use Gr\"{o}nwall's Lemma \ref{13} with $k=1$,  $\gamma_{n}=\frac{\Delta t}{T}$,  $\sigma_{n}=(1-\frac{\Delta t}{T})^{-1}~$, $a_{n}=\parallel E^{n}_{h}\parallel_{0,\Omega}$, $B=0$ 
	and 
	\begin{align*}
		c_{n}&=\Delta t\left(\left(\varepsilon+T\|\bbeta\|_{0,\infty,\Omega}^{2}\right)\lvert  \eta^{n}_{h}\rvert_{1,\Omega}^{2} +T\parallel \delta (i_{h}u^{n})-\partial_{t}u^{n}\parallel_{0,\Omega}^{2}+T\mu^{2}\parallel\eta^{n}_{h} \parallel_{0,\Omega}^{2}+J(\eta^{n}_{h}, \eta^{n}_{h})\right),
	\end{align*}
	which gives
	\begin{align}
		\parallel E^{m}_{h}\parallel_{0,\Omega}^{2}&+\Delta t\sum_{n=1}^{m}\left(\varepsilon\lvert  E^{n}_{h}\rvert_{1,\Omega}^{2}+2\mu\parallel E^{n}_{h}\parallel_{0,\Omega}^{2}+J(E^{n}_{h},E^{n}_{h})\right)\nonumber\\&\leq Ce^{2}\bigg[\Delta t\sum_{n=0}^{m}\bigg\{\left(\varepsilon+T\|\bbeta\|_{0,\infty,\Omega}^{2}\right) \lvert  \eta^{n}_{h}\rvert_{1,\Omega}^{2}+T\parallel \delta (i_{h}u^{n})-\partial_{t} u^{n}\parallel_{0,\Omega}^{2} \label{Grownwall12}\\&\quad+T\mu^{2} \parallel\eta^{n}_{h} \parallel_{0,\Omega}^{2}+J(\eta^{n}_{h}, \eta^{n}_{h})\bigg\}\bigg].\nonumber
	\end{align} 
	
	Next, to reach the error estimate \eqref{error11} we bound each term of the right hand side of \eqref{Grownwall12}. First, using the triangle inequality, yields 
	\begin{align*}
		\parallel \delta(i_{h}u^{n})- \partial_{t} u^{n}\parallel_{0,\Omega}\leq \parallel \delta(i_{h}u^{n})-\delta u^{n}\parallel_{0,\Omega}+\parallel \delta u^{n}-\partial_{t} u^{n}\parallel_{0,\Omega}.
	\end{align*} 
	For the first term in the above inequality, using \eqref{lagranges},  the Taylor's Theorem and the Cauchy-Schwarz inequality, we have
	\begin{align*}
		\parallel  \delta(i_{h}u^{n})-\delta u^{n}\parallel_{0,\Omega}^{2}&\leq Ch^{2k+2}|\delta u^{n} |_{k+1,\Omega}^{2}\\&\leq Ch^{2k+2}\left|\frac{1}{\Delta t}\int_{t_{n-1}}^{t_{n}}\partial_{t} u(t){\rm d}t\right|_{k+1,\Omega}^{2}\\&\leq  C h^{2k+2}\left|\frac{1}{\Delta t}\left(\int_{t_{n-1}}^{t_{n}}{\rm d}t\right)^{\frac{1}{2}}\left(\int_{t_{n-1}}^{t_{n}}|\partial_{t}u(t)|^{2}{\rm d}t\right)^{\frac{1}{2}}\right|_{k+1,\Omega}^{2}\\&\leq C\frac{h^{2k+2}}{\Delta t}\int_{t_{n-1}}^{t_{n}}|\partial_{t}u(t)|_{k+1,\Omega}^{2}{\rm d}t.
	\end{align*} 
	For the second term, one further use of Taylor's Theorem and the Cauchy-Schwarz inequality gives
	\begin{align*}
		\parallel \delta u^{n}-\partial_{t} u^{n}\parallel_{0,\Omega}^{2}&= \int_{\Omega}(\delta u^{n}-\partial_{t} u^{n})^{2}{\rm d}\boldsymbol{x}\\&\leq \int_{\Omega}\left(\int_{t_{n-1}}^{t_{n}}|\partial_{tt} u(t)|{\rm d}t\right)^{2}{\rm d}\boldsymbol{x}\\&\leq \int_{\Omega}\int_{t_{n-1}}^{t_{n}}{\rm d}t\int_{t_{n-1}}^{t_{n}}|\partial_{tt}u(t)|^{2}{\rm d}t~{\rm d}\boldsymbol{x}\\& =\Delta t\int_{t_{n-1}}^{t_{n}}\parallel \partial_{tt}u(t)\parallel_{0,\Omega}^{2}{\rm d}t.
	\end{align*}

	Next,  using \eqref{lagranges}, we have
	\begin{align}
		\lvert  \eta^{n}_{h}\rvert_{1,\Omega}^{2}\leq Ch^{2k}|u^{n} |_{k+1,\Omega}^{2}\quad,\quad \|\eta^n_h\|_{0,\Omega}^2\le C\, h^{2k+2}|u^n|_{k+1,\Omega}^2.\label{inequal32}
	\end{align} 
	Furthermore, using the inverse inequality \eqref{inverse} and the approximation inequality \eqref{lagranges}, we have
	\begin{align} 
		J(\eta_{h}^{n},\eta_{h}^{n})= \gamma \sum_{F\in\calF_I^{}}\int_F\|\bbeta \|_{0,\infty,F}^{} h_{F}^{2}\llbracket\nabla \eta_{h}^{n} \rrbracket \cdot\llbracket\nabla \eta_{h}^{n} \rrbracket \mathrm{d}s\leq 
		C\gamma h^{2k+1}\|\bbeta\|_{0,\infty,\Omega}|u^{n} |_{k+1,\Omega}^{2}.\label{JJJ}
	\end{align}
	
	Gathering all the above bounds, we arrive at 
	\begin{align*}
		\parallel E^{m}_{h}\parallel_{0,\Omega}^{2} & +\Delta t\sum_{n=1}^{m}\bigg(\varepsilon\lvert  E^{n}_{h}\rvert_{1,\Omega}^{2} + 2\mu\parallel E^{n}_{h}\parallel_{0,\Omega}^{2} + J(E^{n}_{h},E^{n}_{h})\bigg)\\
		& \leq Ce^{2}\Bigg[h^{2k}\Bigg\{ \Delta t\sum_{n=0}^{m} \Bigg( \varepsilon  +T\|\bbeta\|_{0,\infty,\Omega}^{2} + h\gamma \|\bbeta\|_{0,\infty,\Omega}+ h^{2}T\mu^{2}\Bigg)|u^{n}|_{k+1,\Omega}^{2} \\
		& \quad+ Th^{2}\sum_{n=0}^{m}\int_{t_{n-1}}^{t_{n}}|\partial_{t} u(t)|_{k+1,\Omega}^{2}{\rm d}t \Bigg\} + T\Delta t^{2}\sum_{n=0}^{m}\int_{t_{n-1}}^{t_{n}}\parallel \partial_{tt}u(t)\parallel_{0,\Omega}^{2}{\rm d}t \Bigg].
	\end{align*}
	Finally, using the triangle inequality and (\ref{lagranges}) once again, we arrive at the final estimate
	\begin{align*}
		\max\limits_{1\leq m\leq N}&\parallel E^{m}\parallel_{0,\Omega}^{2} + \Delta t\sum_{n=1}^{N}\bigg(\varepsilon\lvert  E^{n}_{h}\rvert_{1,\Omega}^{2} + \mu\parallel E^{n}_{h}\parallel_{0,\Omega}^{2} + J(E^{n}_{h},E^{n}_{h})\bigg)\\
		& \leq Ce^{2}\Bigg[h^{2k}\Bigg\{\Delta t \sum_{n=0}^{N} \Bigg(\varepsilon + T\|\bbeta\|_{0,\infty,\Omega}^{2}  +h\gamma \|\bbeta\|_{0,\infty,\Omega}+ h^{2}T\mu^{2}\Bigg)|u^{n} |_{k+1,\Omega}^{2}\\
		& \quad + h^{2}\max\limits_{1\leq m\leq N}|u^{m} |_{k+1,\Omega}^{2} + Th^{2}\int_{0}^{T}|\partial_{t} u(t)|_{k+1,\Omega}^{2}{\rm d}t\Bigg\}  + T\Delta  t^2\int_{0}^{T}\parallel \partial_{tt} u(t)\parallel_{0,\Omega}^{2}{\rm d}t\Bigg],
	\end{align*}
	which proves the result.
\end{proof}

\section{Numerical experiments}	\label{sec:numerics}
In this section we present two experiments to test the numerical
performance of \eqref{eq199}. In these experiments we have used $\Omega=(0,1)^2$, and 
the value $\alpha=1$ in the stabilising bilinear form
$s(\cdot,\cdot)$. Except for the very last numerical result in
this section, we have used two types of meshes, a
three-directional triangular mesh and a regular quadrilateral one.
The coarsest level of each is depicted in Figure~\ref{meshs}.
 \begin{figure}[h!]
   \centering
   \subfloat[A symmetric, Delaunay mesh. ]{\label{e11}
     \includegraphics[width=0.23\textwidth]{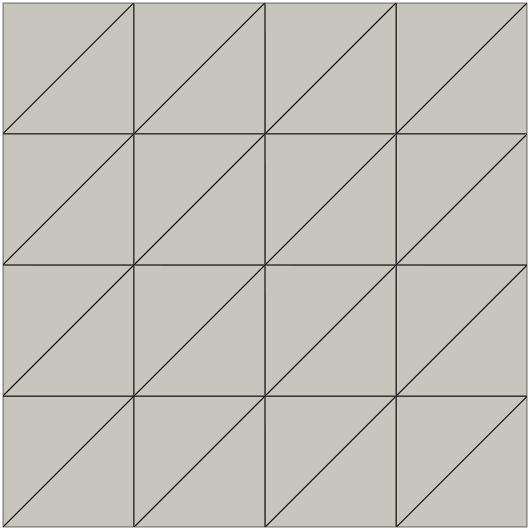}
   }	\hspace{20pt}
   \subfloat[A simple quadrilateral mesh.]{\label{Q11}
     \includegraphics[width=0.23\textwidth]{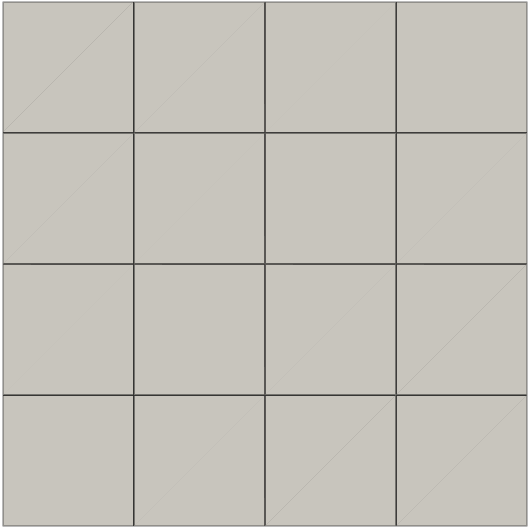}
   }\hspace{20pt}
\subfloat[A non-symmetric, non-Delaunay mesh. ]{\label{ND11}
	\includegraphics[width=0.23\textwidth]{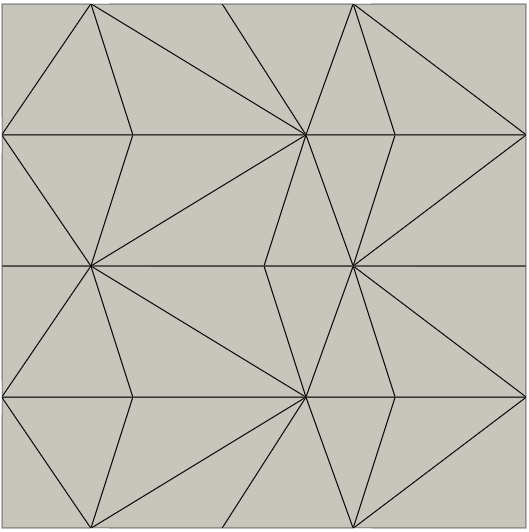}
}	
   \caption{Three coarse level indicative meshes used in the
     experiments all with $N=5$.}\label{meshs}
 \end{figure} 
 
To solve the nonlinear problem \eqref{eq199} at each discrete time
$t_{n}$, $n=1,2,\ldots,N$, first we set $\tilde{u}^{0}:=u_{h}^{n-1}$.
Next, by choosing an appropriate damping parameter $\omega\in(0,1]$
  the following fixed point Richardson-like iterative method is used
  to find $\tilde{u}^{m+1}\in V_{\mathcal{P}}$, such that
  \begin{align}
    &	\left(\tilde{u}^{m+1},v_{h}\right)_{\Omega}+ \Delta t\theta a_{J}\left(\tilde{u}^{m+1},v_{h}\right) = \left(\tilde{u}^{m},v_{h}\right)_{\Omega} + \Delta t\theta a_{J}\left(\tilde{u}^{m},v_{h}\right) \label{iter1} \\
    &\hspace{1.0cm} + \omega \left\{F^n(v_{h}) - \left[\left((\tilde{u}^{m})^{+},v_{h}\right)_{\Omega} + \Delta t\theta a_{J}\left((\tilde{u}^{m})^{+},v_{h}\right)+ s\left((\tilde{u}^{m})^{-},v_{h}\right)\right]\right\}  \quad\forall v_{h}\in V_{\mathcal{P}},\nonumber
  \end{align}
  for $m=1,2,\ldots,N_{\textrm{max}}^{}$, or until the following stopping criterion is achieved
\begin{align}
  \parallel \tilde{u}^{m+1}-\tilde{u}^{m} \|_{0,\Omega}\leq 10^{-8}. \label{stoping}
\end{align}
 Finally,  for $m_{0}$ which satisfies in \eqref{stoping},  we set $u_{h}^{n}=\tilde{u}^{m_{0}+1}$.
 
%%%%%%%%%%%%%%%%%%%%%%%%%%%%%%%%%%%%%%%%%%%%%%%%%%%%%%%%%%%%%%%%%%%%%%%%%%%%%%%%%%%%%%%%%%%%%%%%%%%%%%%%%%%%%%%%%%%%%%%%%%%%%%%%

In all figures, $P-1$ indicates the number of divisions in the $x$ and
$y$ directions, resulting in a total of $P^2$ vertices, including the
boundary. We evaluate the method's asymptotic performance in the
$\parallel \cdot \parallel_{0,\Omega}$-norm at the final step, i.e.,
$\parallel e_{h}^{N}\parallel_{0,\Omega}$, and to verify the result
from Theorem \ref{Theorem11} we examine the asymptotic behaviour of
the error by the following norm
\begin{eqnarray} 
\parallel e^{N}\parallel_{h}^{2}:=\parallel e^{N}\parallel_{0,\Omega}^{2}+\sum_{n=0}^{N}\Delta t\left(\varepsilon\parallel \nabla e^{n}_{h}\parallel_{0,\Omega}^{2}+\mu\parallel  e^{n}_{h}\parallel_{0,\Omega}^{2}+J(e^{n}_{h},e^{n}_{h})\right). \label{norm22}
\end{eqnarray}
We have used $\mathbb{P}_1^{}$ and $\mathbb{P}_2^{}$ elements in the
triangular meshes, and $\mathbb{Q}_1^{}$ elements in the quadrilateral
mesh. In the numerical experiments we use the bound preserving Euler
(BP-Euler) \eqref{Euler} and the bound preserving Crank-Nicholson
(BP-CN), i.e., the method \eqref{eq199} with $\theta=\frac{1}{2}$,
even though stability and error estimates for BP-CN has not been
proven.  
\begin{example}[A problem with a smooth solution]\label{Example1}
  We consider $\mu=1$, $\varepsilon=10^{-6}$,
  $\boldsymbol{\beta}=(2,1)$, and set $f$ and $u^{0}$ such that the
  function
   \begin{align*}
  	u(x,y,t)=\exp(t)\sin(\pi x)\sin(\pi y)\, \quad \Omega=(0,1)^2,
  	\end{align*}
  	 is the
  analytical solution of (\ref{CDR}).  Notice that $u(x,y,t)\in
  [0,\exp(t)]$, and thus we set $\kappa(t)=\exp(t)$ as the upper bound
  at time $t$.  The CIP stabilisation parameter $\gamma=0.05$ has been
  used in \eqref{eq10} and we set $\omega=0.1$ for the damping
  parameter in all the time steps.
	
  Figure \ref{Fig121} illustrates the asymptotic behaviour of the
  error $\parallel e_{h}^{N}\parallel_{0,\Omega}$ using $\mathbb{P}_1$
  and $\mathbb{P}_2$ elements. These results align with the
  theoretical findings we established in Theorem \ref{Theorem11}. By
  fixing $\Delta t=4\times10^{-4}$ and decreasing the mesh size as
  depicted in Figures \ref{Eulerh1243} and \ref{Eulert}, we observe
  second- and third-order convergence when using $\mathbb{P}_1$
  and $\mathbb{P}_2$ elements, respectively, for both BP-Euler and
  BP-CN. Also, by fixing the mesh size $h=5\times10^{-3}$ and varying
  the length of the time-step $\Delta t$, as shown in Figures
  \ref{CNh12556} and \ref{BPMCN12}, we obtained first-order
  convergence for the Euler method and second-order convergence for
  Crank-Nicholson for both $\mathbb{P}_1$ and $\mathbb{P}_2$ elements,
  as expected.
  
  Figure \ref{Fig12132} depicts the asymptotic behaviour of the error
  $\parallel e^{N}\parallel_{h}^{2}$ using $\mathbb{P}_1$ and
  $\mathbb{P}_2$ elements. These results also corroborate the
  theoretical results we proved in Theorem \ref{Theorem11}. By fixing
  the time step and decreasing the mesh size as shown in Figures
  \ref{Eulerh12431} and \ref{Eulert1}, we observed second- and
  third-order convergence when using $\mathbb{P}_1$ and $\mathbb{P}_2$
  elements, respectively for the BP-Euler method.  This extra order of
  convergence is, most likely, due to the fact that the small value of
  $\varepsilon$ makes that the $\|\cdot\|_h$ norm is dominated by the
  $L^2(\Omega)$-norm.  Additionally, we achieved first-order
  convergence for the BP-Euler method for both $\mathbb{P}_1$ and
  $\mathbb{P}_2$ elements when the size of the time step is decreased.
  Figures \ref{CNh125561} and \ref{BPMCN121} show the asymptotic
  behaviour with respect to time for the BP-Euler method.
  
  To assess the computational cost of the nonlinear algorithm at each time step, we depicted in Figure \ref{Average1} the average number of iterations per step over 1000 time steps for a sequence of meshes with decreasing mesh size. The results indicate that there is no significant increase in the average number of iterations, which remains relatively low regardless of the mesh size.
\end{example}
\begin{figure}[H]
  \subfloat[$\Delta t=4\times10^{-4}$, T=0.2, $\mathbb{P}_1^{}$ elements.]{\label{Eulerh1243}
    \includegraphics[width=0.53\textwidth]{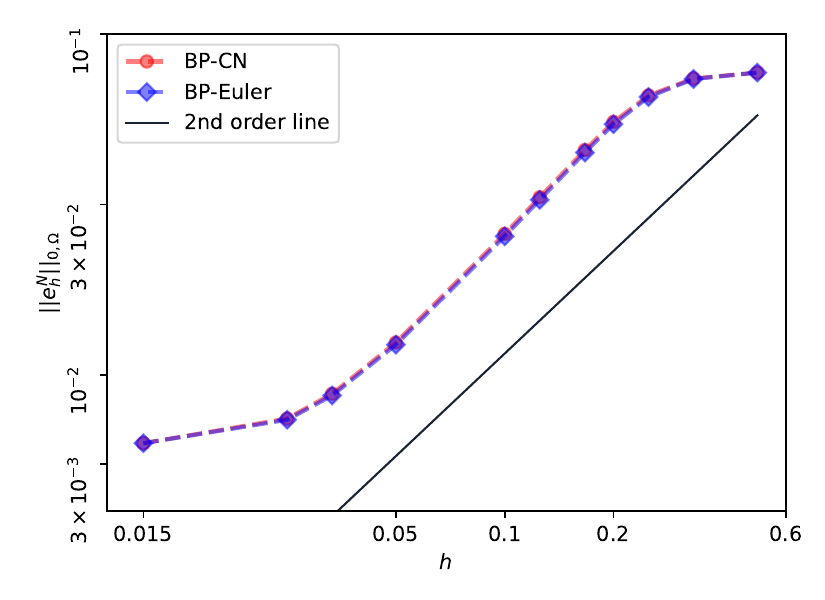}
  }
  \subfloat[$h=5\times10^{-3}$, $T=1$, $\mathbb{P}_1^{}$ elements.]{\label{CNh12556}
    \includegraphics[width=0.53\textwidth]{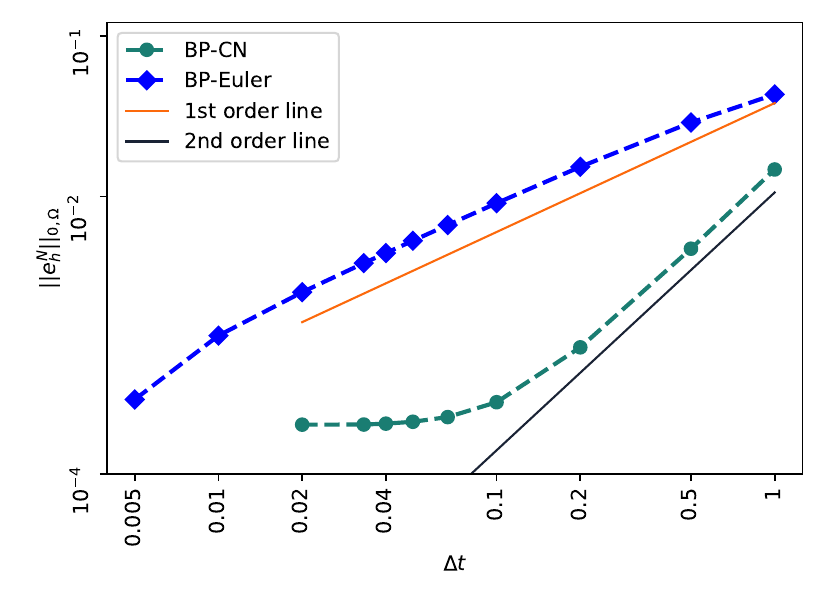}
  }\\
  \subfloat[$\Delta t=4\times10^{-4}$, T=0.2, $\mathbb{P}_2^{}$ elements.]{\label{Eulert}
    \includegraphics[width=0.53\textwidth]{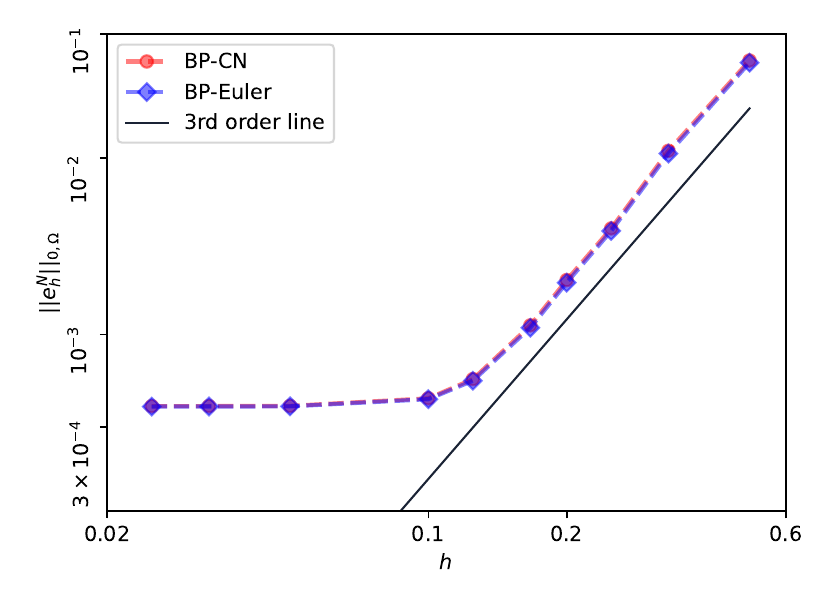}
  }
  \subfloat[$h=5\times10^{-3}$, $T=1$, $\mathbb{P}_2^{}$ elements]{\label{BPMCN12}
    \includegraphics[width=0.53\textwidth]{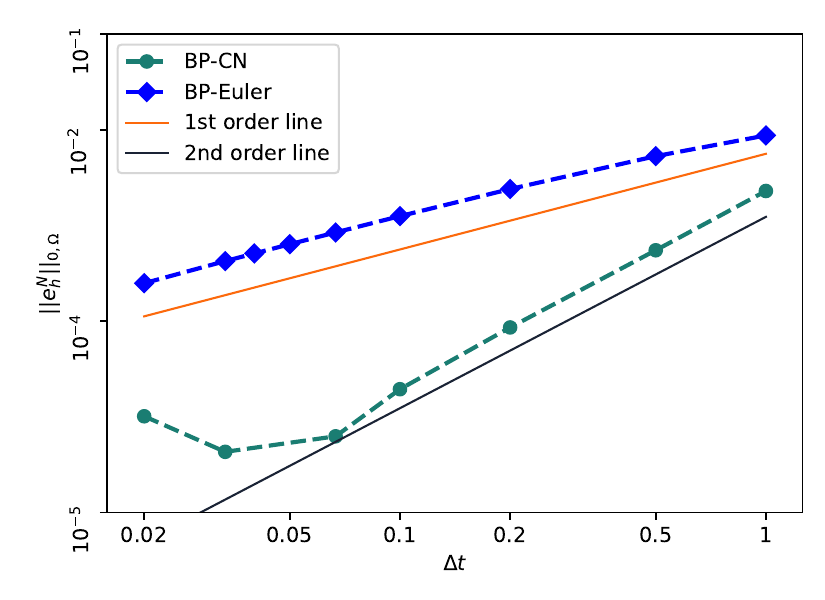}
  }
  \caption{ Comparison of the error of the approximated solution by  the BP-Euler method and BP-CN method with the exact solution in $\parallel \cdot\parallel_{0,\Omega}$-norm (using mesh
  	\ref{e11}). }\label{Fig121}
\end{figure}  
%%%%%%%%%%%%%%%%%%%%%%%%%%%%%%%%%%%%%%%%%%%%%%%%%%%%%%%%%%%%%%%%%%%%%%%%%%%%%%%%%%%%%%%%%%%%%%%%%%%%%%%%%%%%%%%%
\begin{figure}[H]
  \subfloat[$\Delta t=4\times10^{-4}$, T=0.2, $\mathbb{P}_1^{}$ elements.]{\label{Eulerh12431}
    \includegraphics[width=0.53\textwidth]{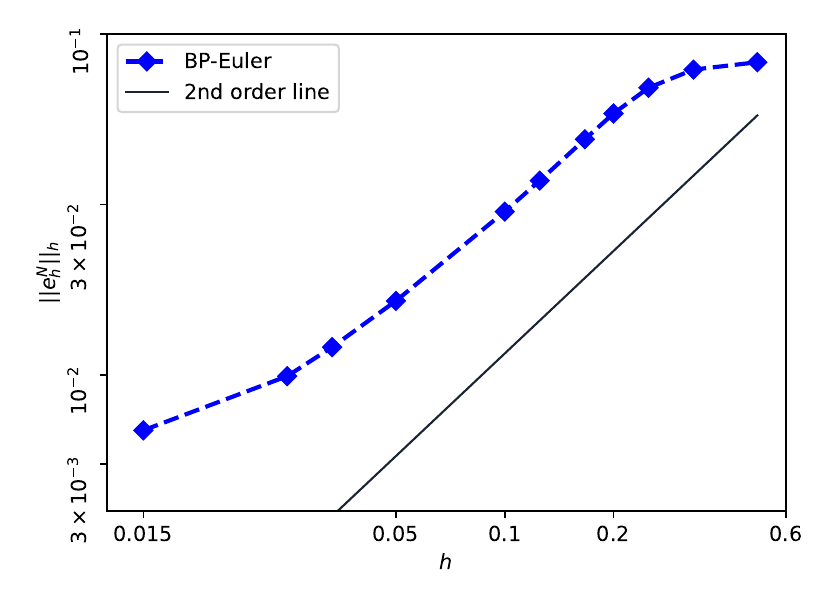}
  }
  \subfloat[$h=5\times10^{-3}$, $T=1$, $\mathbb{P}_1^{}$ elements.]{\label{CNh125561}
    \includegraphics[width=0.53\textwidth]{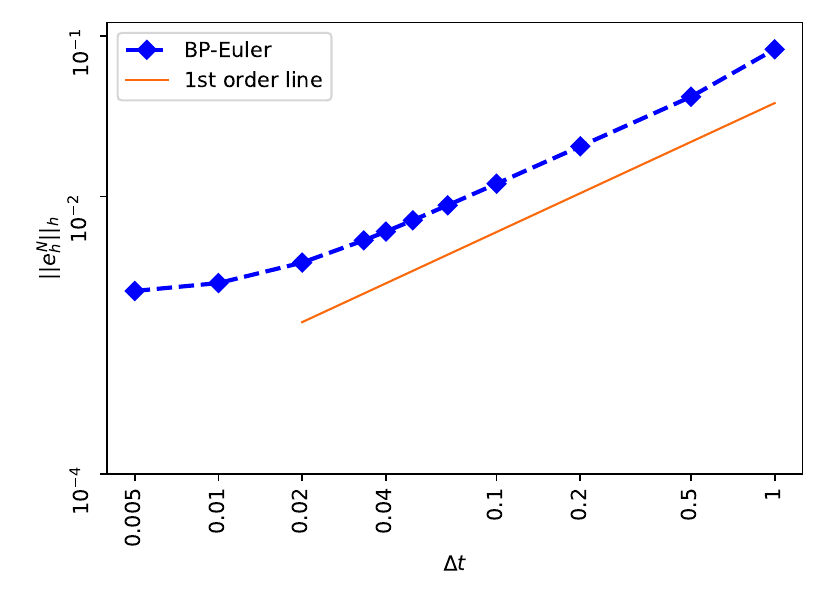}
  }\\
  \subfloat[$\Delta t=4\times10^{-4}$, T=0.2, $\mathbb{P}_2^{}$ elements.]{\label{Eulert1}
    \includegraphics[width=0.53\textwidth]{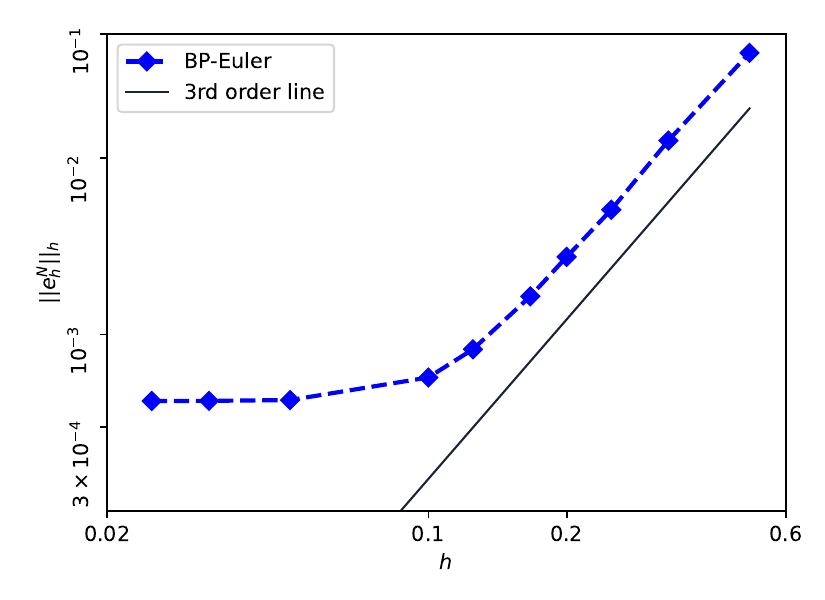}
  }
  \subfloat[$h=5\times10^{-3}$, $T=1$, $\mathbb{P}_2^{}$ elements]{\label{BPMCN121}
    \includegraphics[width=0.53\textwidth]{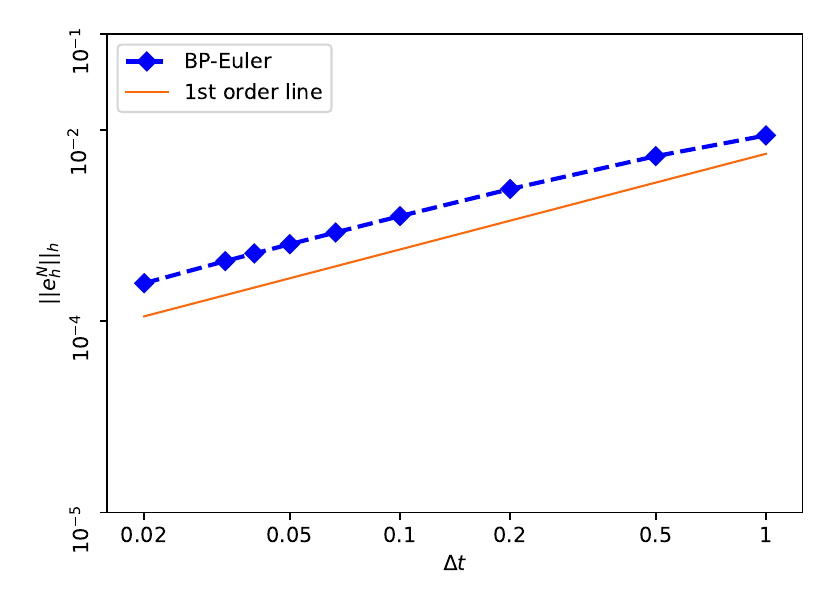}
  }
  \caption{ Using norm (\ref{norm22}) for the comparison of the error of the approximated solution by the BP-Euler method with the exact solution (using mesh
  	\ref{e11}).}\label{Fig12132}
\end{figure}  
\begin{figure}[H]
	\subfloat[$\mathbb{P}_1^{}$ elements]{
		\includegraphics[width=0.53\textwidth]{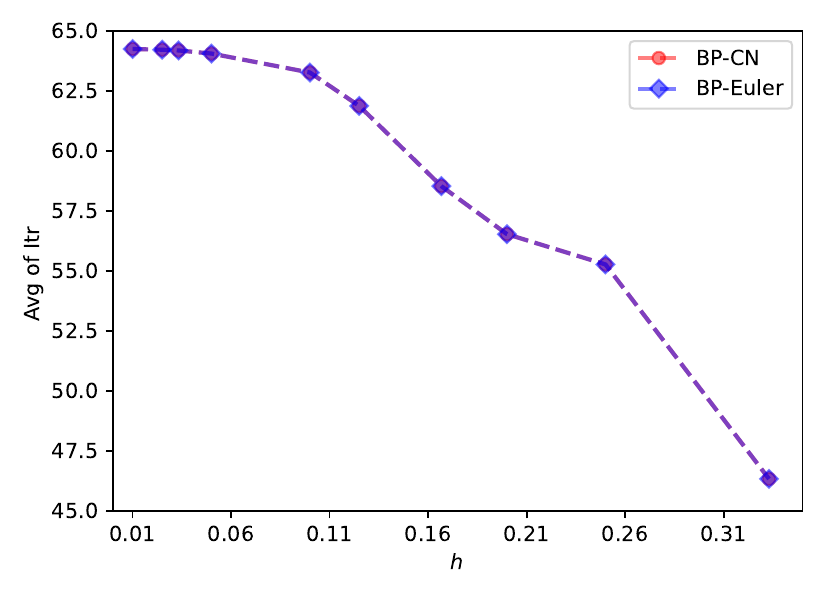}
	}
	\subfloat[ $\mathbb{Q}_1^{}$ elements]{
		\includegraphics[width=0.53\textwidth]{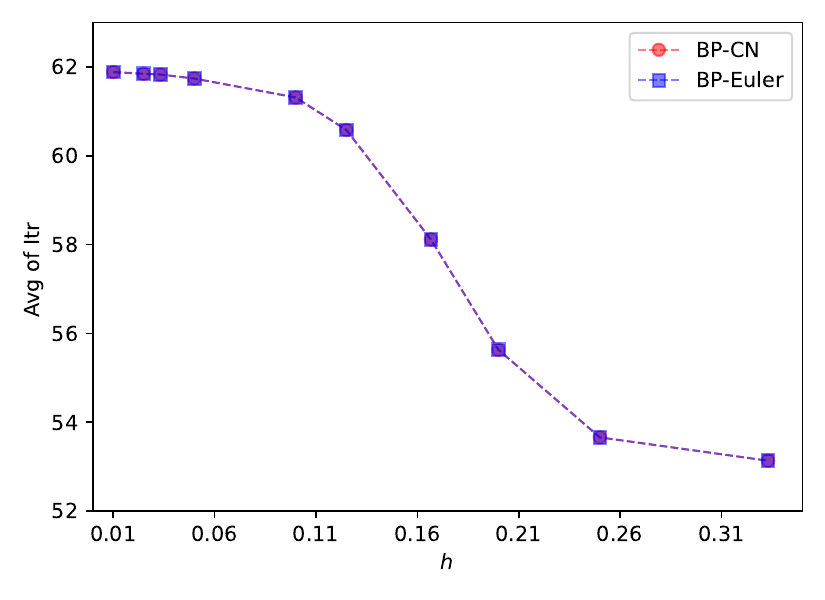}
	}
	\caption{The average number of the Richardson iterations of 1000 time steps ($T=1$) needed to reach convergence using $\mathbb{P}_{1}$ and $\mathbb{Q}_{1}$ elements and BP-Euler and BP-CN  methods and the meshes
		\ref{e11} and \ref{Q11}.}\label{Average1}
\end{figure}  
\begin{example}[Three body rotation]\label{Example2}
  This example is a modified version of the three body rotation
  transport problem from \cite{MR1388492}.  We used
  $\boldsymbol{\beta}=(0.5-y,x-0.5)$, $\varepsilon=10^{-12}$ and
  $\mu=f=0$. The initial setup involves three separate bodies, as
  depicted in Figure \ref{Initial123}. Each body's position is defined
  by its centre at coordinates $(x_0, y_0)$. Every body is contained
  within a circle of radius $r_0 = 0.15$ cantered at \((x_0,
  y_0)\). Outside these three bodies, the initial condition is zero.

Let  
\begin{eqnarray}
	r(x,y)=\frac{1}{r_{0}}\sqrt{(x-x_{0})^{2}+(y-y_{0})^{2}}.
\end{eqnarray}
The center of the slotted cylinder is in $(x_{0},y_{0})=(0.5,0.75)$ and its geometry is given by
\begin{align}\label{eq2}
	u(0;x,y)=\left\{ \begin{array}{ll}1 \hspace{0.3cm}{\rm if}\hspace{0.3cm}r(x,y)\leq 1, |x-x_{0} |\geq 0.0225 \hspace{0.3cm} {\rm or} y\geq 0.85;\\  0 \hspace{0.2cm} {\rm else},\end{array} \right.
\end{align}
The conical body at the bottom side is described by $(x_{0},y_{0})=(0.5,0.25)$ and 
\begin{eqnarray}
	u(0;x,y)=1-r(x,y).
\end{eqnarray}
Finally, the hump at the left hand side is given by $(x_{0},y_{0})=(0.25,0.5)$ and 
\begin{eqnarray}
	u(0;x,y)=\frac{1}{4}(1+\cos(\pi \min\{r(x,y),1\})).
\end{eqnarray}
The rotation of the bodies occurs counter-clockwise. A full revolution takes $t=2\pi$. We use $P=130$ so a regular grid consisting of $130\times 130$ mesh cells for $\mathbb{P}_1^{}$, $\mathbb{P}_2^{}$, and $\mathbb{Q}_1^{}$ elements. 
\end{example}
%%%%%%%%%%%%%%%%%%%%%%%%%%%%%%%%%%%%%%%%%%%%%%%%%%%%%%%%%%%%%%%%%%%%%%%% 
\begin{figure}[h!]
	\centering
	\includegraphics[width=0.6\textwidth]{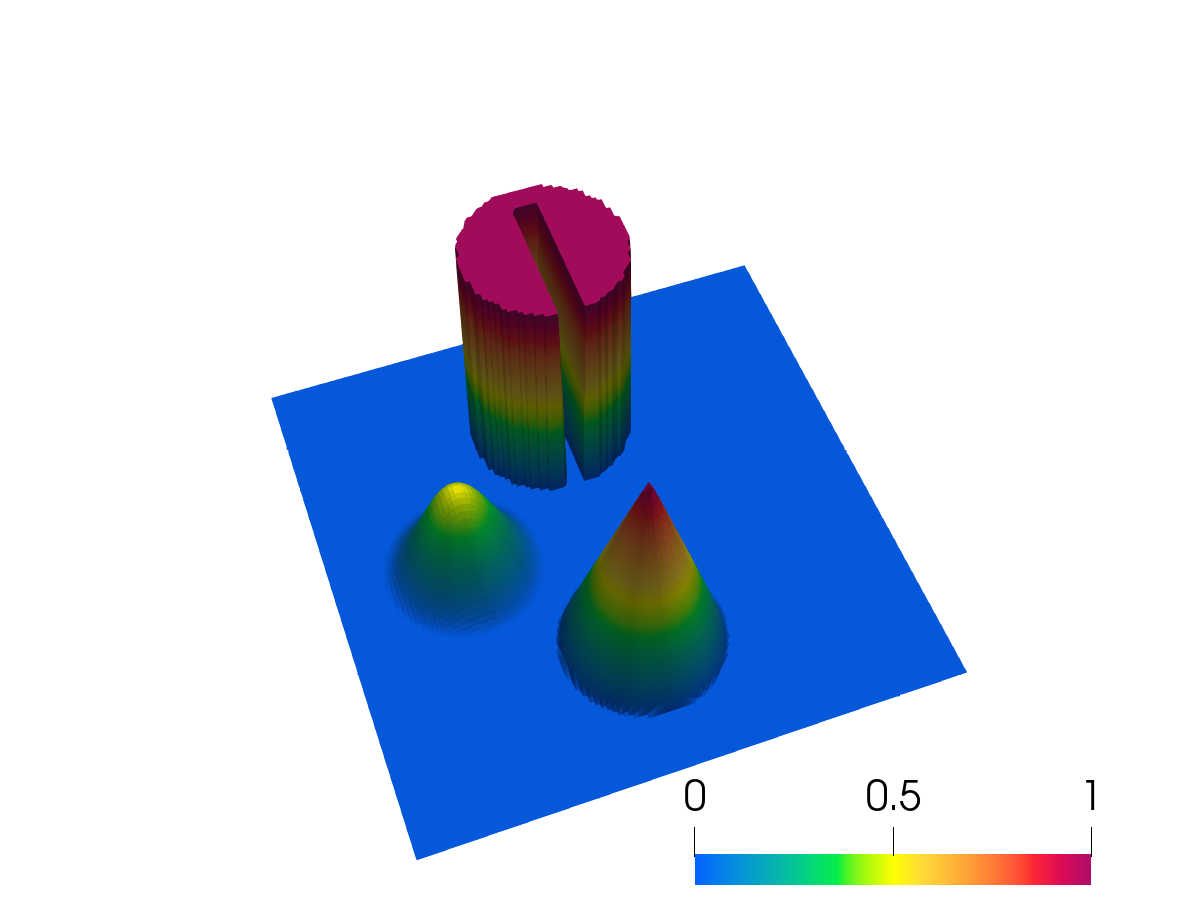}
	\caption{Initial data $u^{0}$ for rotating body problem.}\label{Initial123}
\end{figure}
 \begin{figure}[h!]
 	\centering
 	\subfloat[BP-Euler-$\mathbb{P}_1^{}$ elements]{
 		\includegraphics[width=0.53\textwidth]{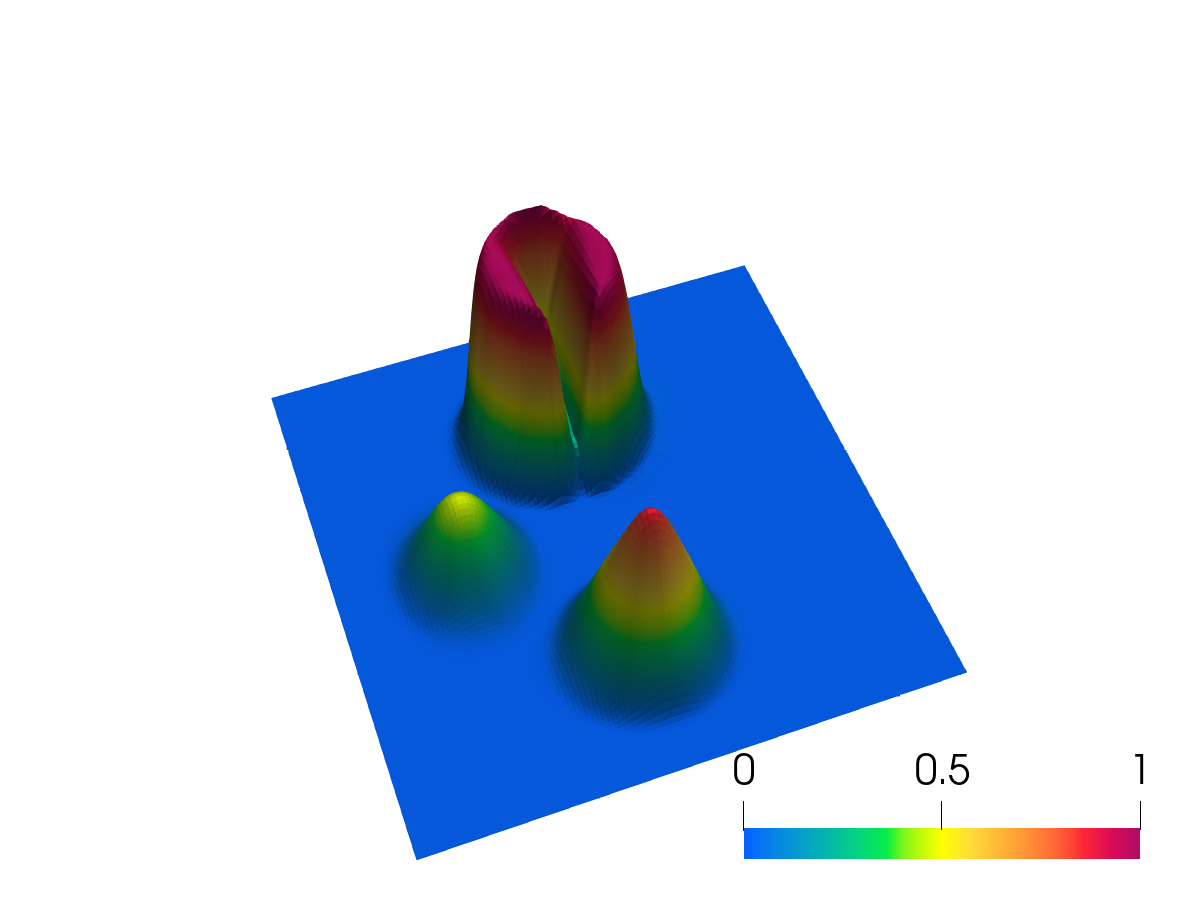}
 	}
 	\subfloat[BP-CN,$\mathbb{P}_1^{}$ elements]{
 		\includegraphics[width=0.53\textwidth]{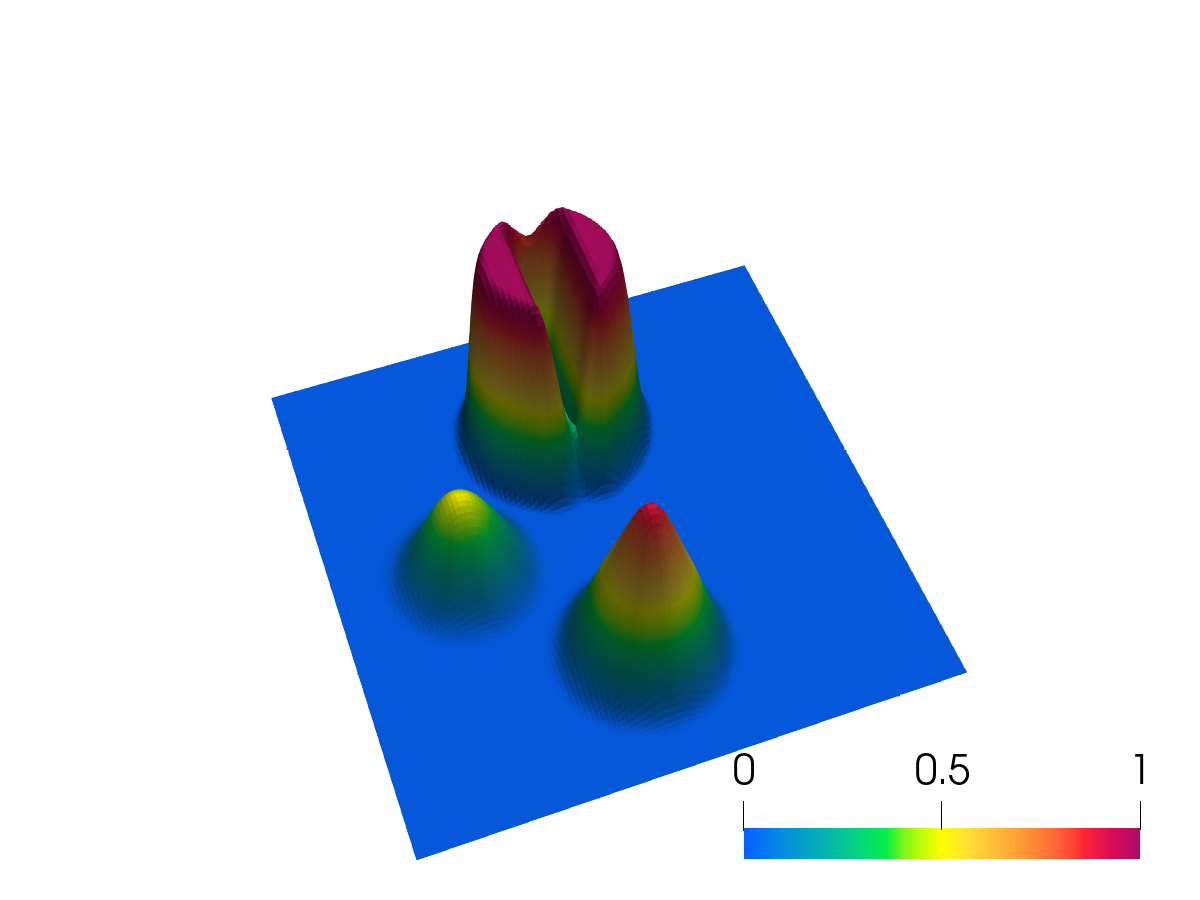}
 	}\\
 	\subfloat[BP-Euler-$\mathbb{P}_2^{}$ elements]{
 		\includegraphics[width=0.53\textwidth]{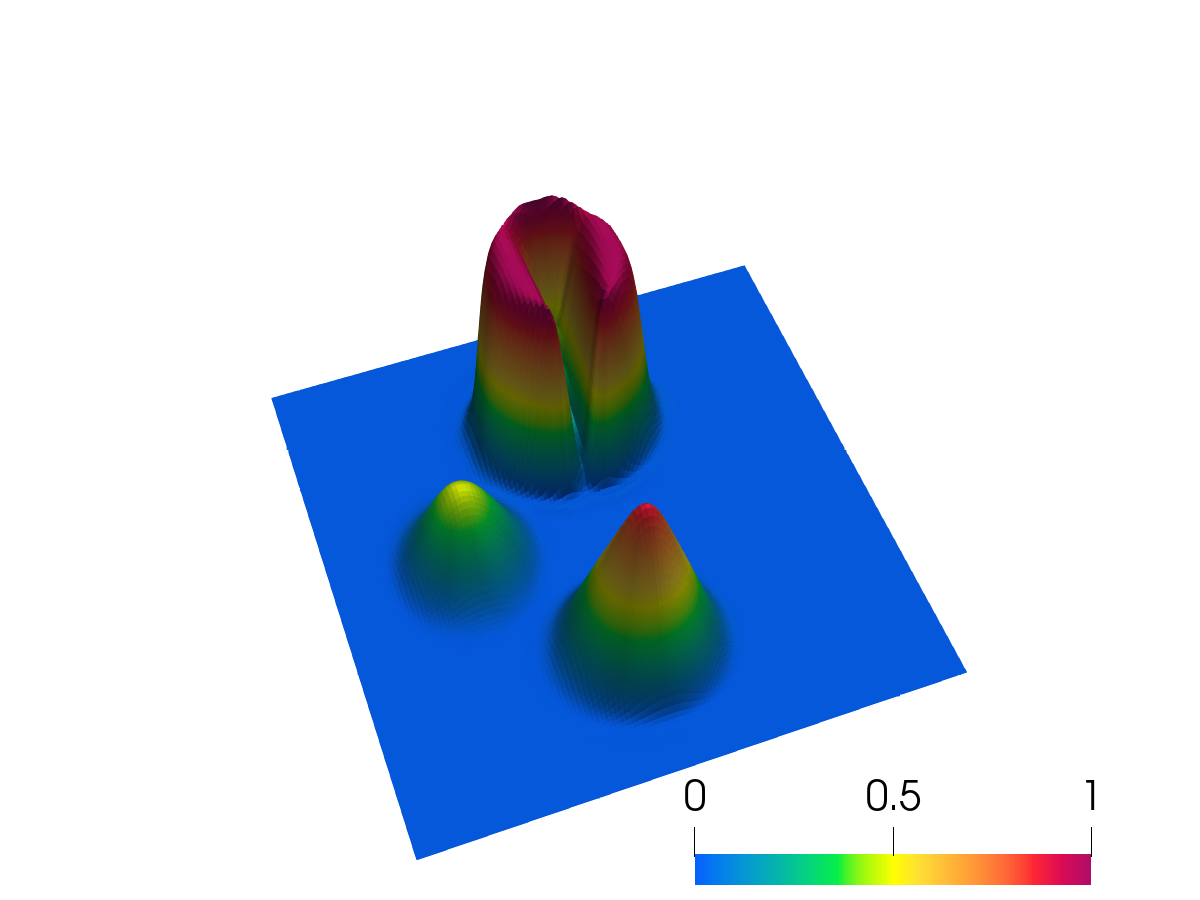}
 	}
 	\subfloat[BP-CN-$\mathbb{P}_2^{}$ elements]{
 		\includegraphics[width=0.53\textwidth]{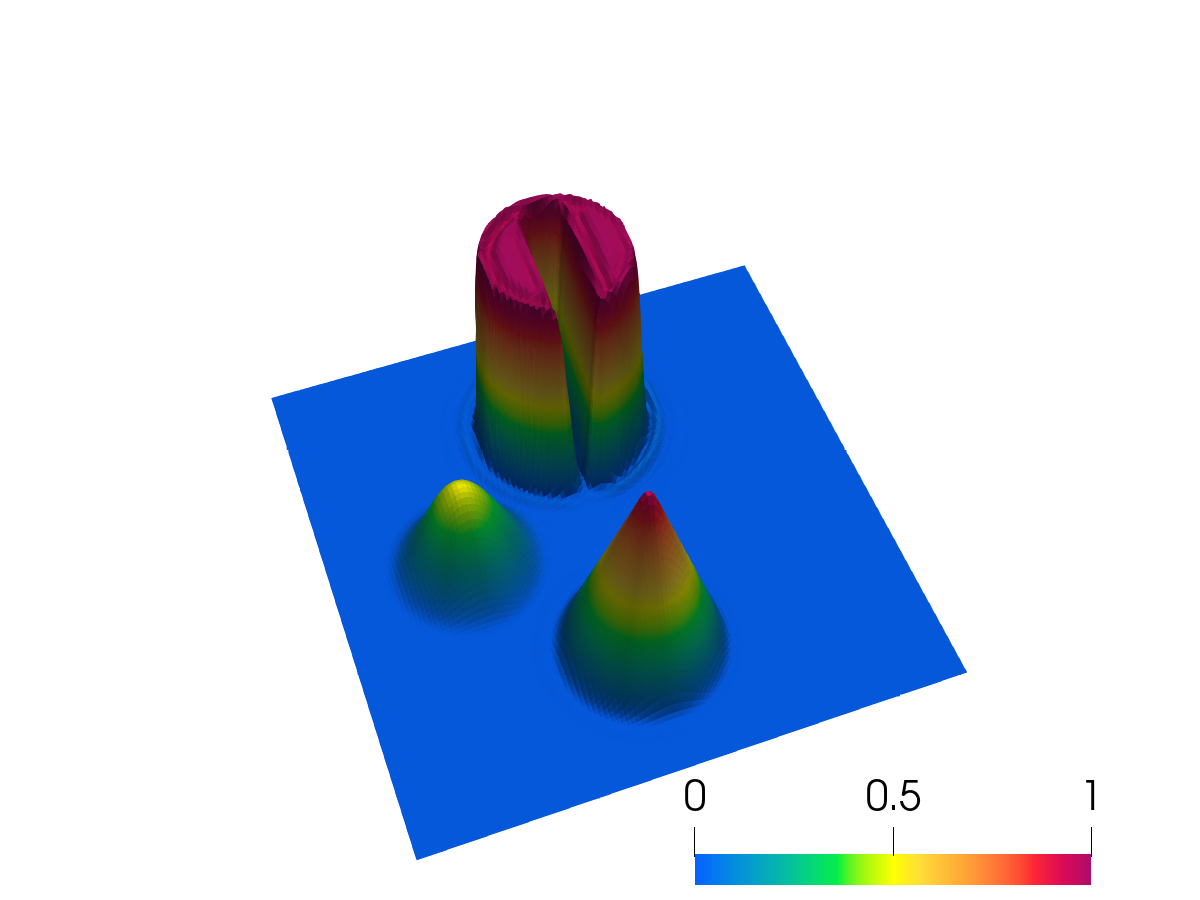}
 	}\\
 	\subfloat[BP-Euler-$\mathbb{Q}_1^{}$ elements]{
 		\includegraphics[width=0.53\textwidth]{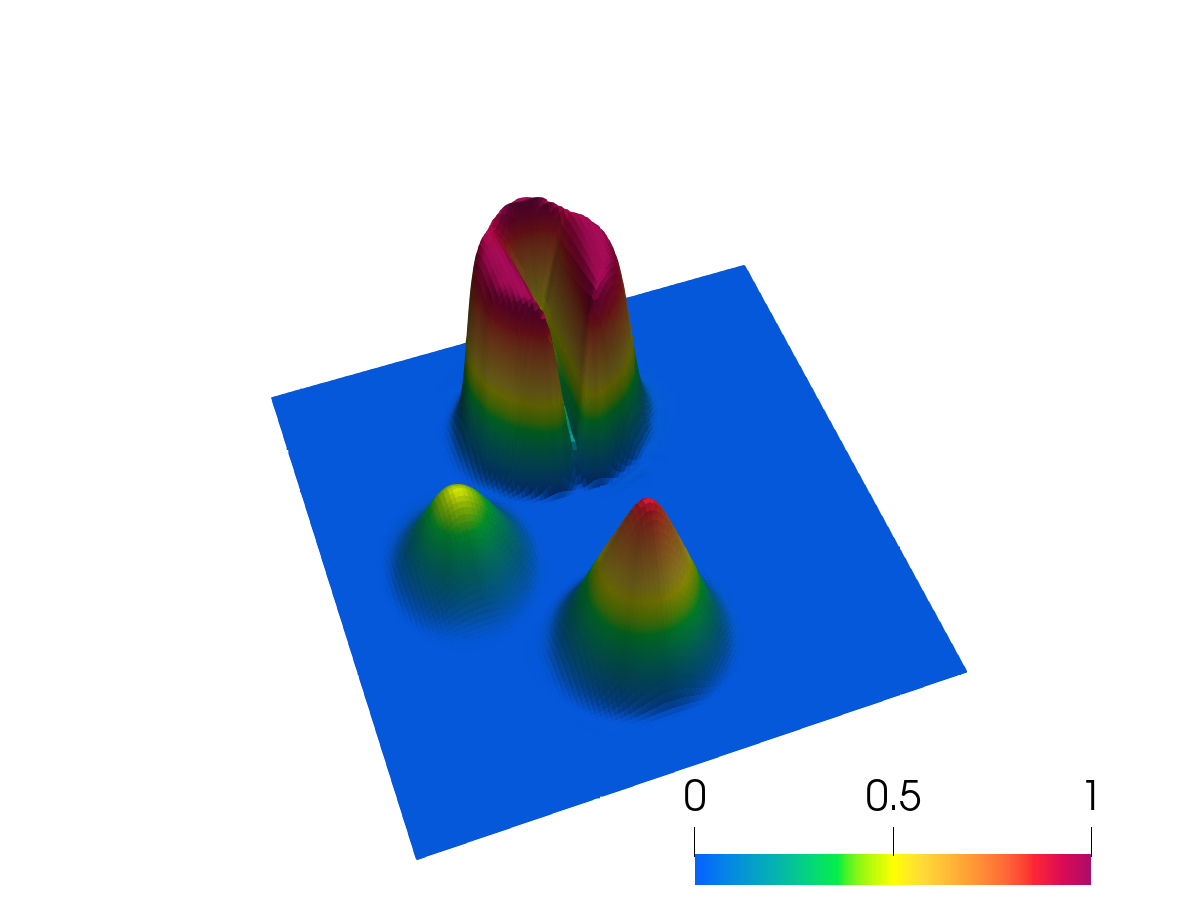}
 	}
 	\subfloat[BP-CN-$\mathbb{Q}_1^{}$ elements]{
 		\includegraphics[width=0.53\textwidth]{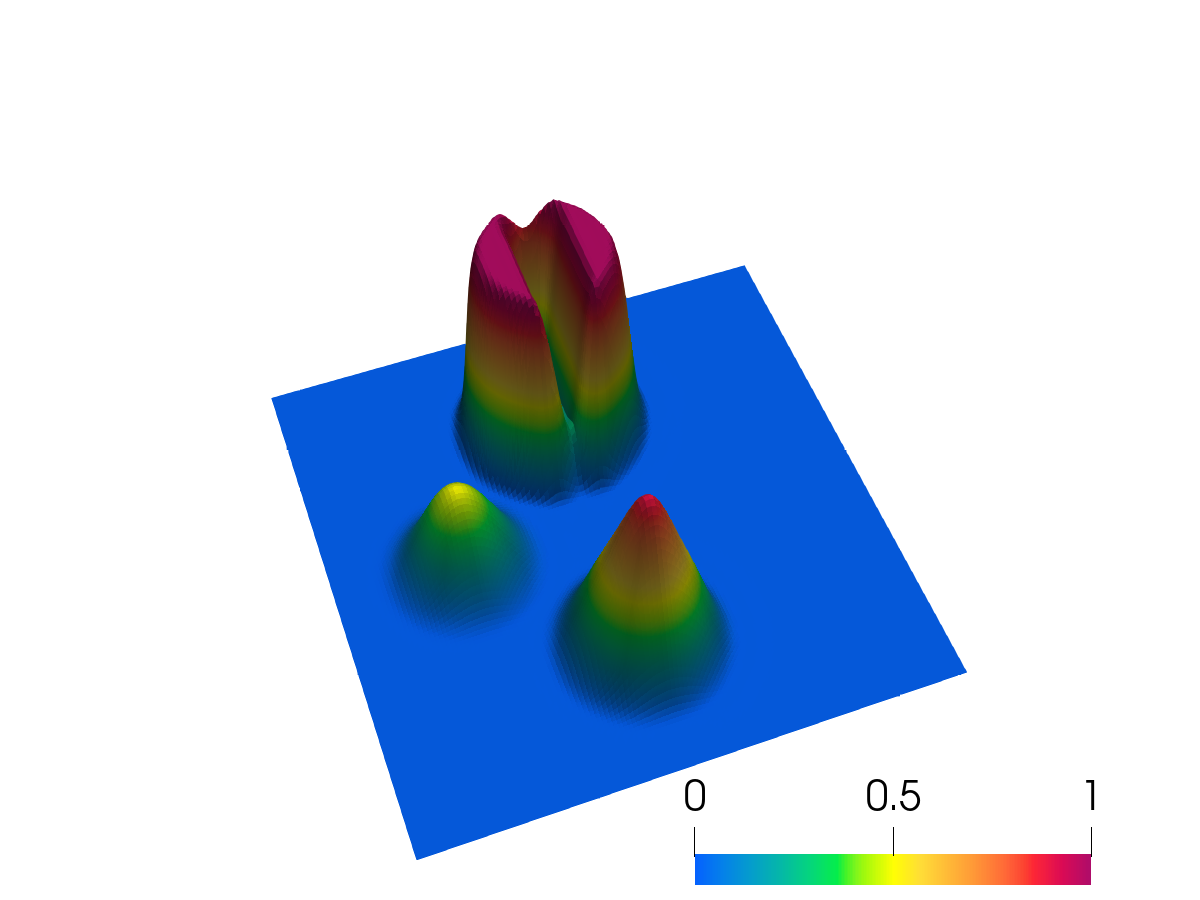}
 	}
 	\caption{  The approximation of the solution of Example 2 for BP-Euler method and BP-CN method at $T=6.28$ ($\gamma=0.001$, $P=130$).  }\label{Fig143}
 \end{figure} 
%%%%%%%%%%%%%%%%%%%%%%%%%%%%%%%%%%%%%%%%%%%%%%%%%%%%%%%%%%%%%%%%%%%%%%%% 
%%%%%%%%%%%%%%%%%%%%%%%%%%%%%%%%%%%%%%%%%%%%%%%%%%%%%%%%%%%%%%%%%%% Figure 1  
%%%%%%%%%%%%%%%%%%%%%%%%%%%%%%%%%%%%%%%%%%%%%%%%
The simulations were performed with the final time $T=2\pi$ and the time step $\Delta t=10^{-3}$. Figure \ref{Fig143} depicts the approximation solution for the BP-Euler method and BP-CN method using $\mathbb{P}_1^{}$, $\mathbb{P}_2^{}$ and $\mathbb{Q}_1^{}$ elements. In both methods the CIP term (\ref{eq10})  was used with the parameter $\gamma=0.001$. 
Our numerical experiments show that the optimal value of $\omega$ (relative to the number of iterations needed to reach convergence) is approximately  $0.07$ 
when using the quadrilateral mesh, while for $\mathbb{P}_1$ and $\mathbb{P}_2$ elements, it is around $0.12$. So, we report the results using these values.

For comparison purposes, we also approximated the same problem with CIP-Euler and  CIP-Crank-Nicolson (CIP-CN) (the method that only use the CIP term i.e., the full time-space discretisation $\theta$ scheme of the method \eqref{eq824}) with the same value for the parameter $\gamma$. 
To compare the numerical solution of different methods, a cross section along the line $y=0.75$ was taken of $u^{0}$, BP-Euler, BP-CN, CIP-Euler and CIP-CN methods. The results have been shown in Figure \ref{Cross123}. Since we perform a full rotation, we can compare the solution from each method with the initial condition to assess the diffusive properties of each method. As shown in Figure \ref{Cross123}, among all the methods, BP-CN demonstrates the best performance regardless of the type of elements used. 

Two final experiments are presented. The first aims at  assessing the effect of adding CIP stabilization to the method \eqref{eq199}. For this, we set $\gamma=0$ in $J(\cdot,\cdot)$
 and the results are shown in Figure \ref{FigurenonCIP} for the BP-CN method using $\mathbb{P}_1^{}$
 elements. In Figure \ref{FigurenonCIP} we can observe the solution 
$u_{h}^{+}$, while respecting the bounds of the exact solution, exhibits spurious oscillations near the layers. This justifies the need for CIP stabilization in \eqref{eq199}.

To test the performance of the method in the case when the mesh used is not Delaunay, we have approximated this example also in the non-Delaunay mesh
depicted in Figure~\ref{ND11}, using $P=130$. The same cross-sections of the approximate solutions for the BP-Euler and BP-CN methods are depicted in Figure~\ref{Fig:Ex2-NonDel}, alongside the 
cross-sections for the CIP stabilised finite element method. In both cases $\gamma=0.001$ has been used in the simulation.  From the results we can observe, once again, that $u_h^+$
respects the bounds of Assumption~(A1), while the CIP solution presents noticeable over and undershoots.

Finally, to study mass conservation we use the relative mass, i.e. the ratio of the mass at time $t$ to the initial mass defined by
\[
M_{r}(t)=\frac{M(t)}{M(0)},
\]
where $M(t)$ is the total mass at time \( t \), and is defined as
\[
M(t) = \int_{\Omega} u(\boldsymbol{x},t) \, \mathrm{d}\boldsymbol{x}.
\]

The evolution of mass over time for BP-Euler and BP-CN methods is presented in Figure \ref{Fig:MASS}. 
The plot depicts the evolution of the relative mass. In there, we observe that, despite the fact that the scheme does not preserve mass, the mass loss/gain remains low throughout the simulation. 

\begin{figure}[h!]
	\centering
	\subfloat[Euler, $\mathbb{P}_1^{}$ elements]{
		\includegraphics[width=0.46\textwidth]{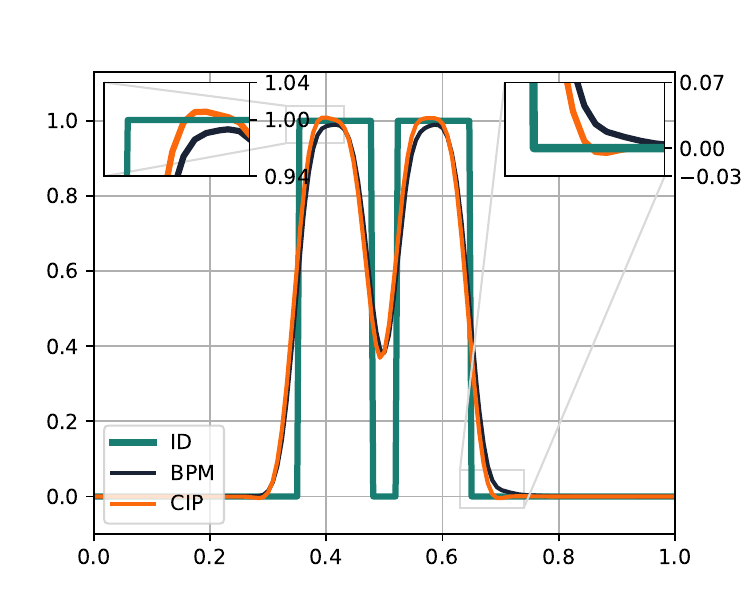}
	}
		\subfloat[CN, $\mathbb{P}_1^{}$ elements]{
			\includegraphics[width=0.46\textwidth]{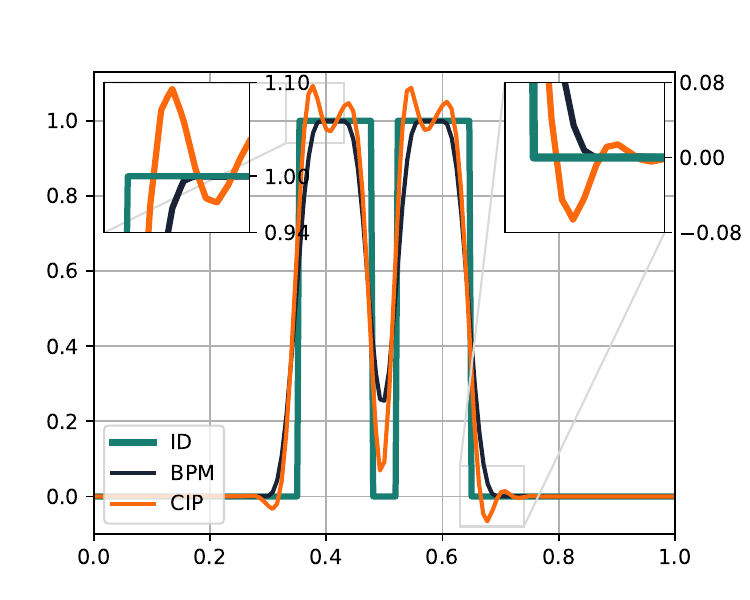}
		}\\
	\subfloat[Euler, $\mathbb{P}_2^{}$ elements]{
		\includegraphics[width=0.46\textwidth]{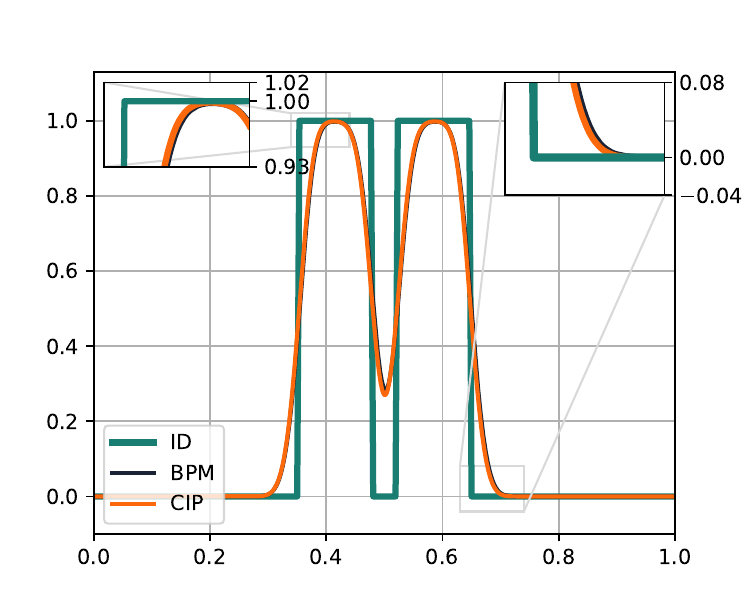}
	}
	\subfloat[CN, $\mathbb{P}_2^{}$ elements]{
		\includegraphics[width=0.46\textwidth]{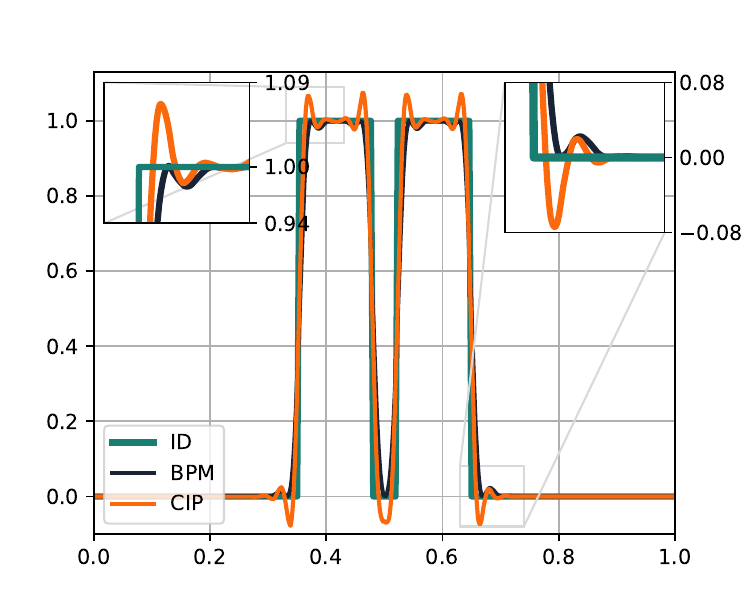}
	}
	\\
	\subfloat[Euler, $\mathbb{Q}_1^{}$ elements]{
		\includegraphics[width=0.46\textwidth]{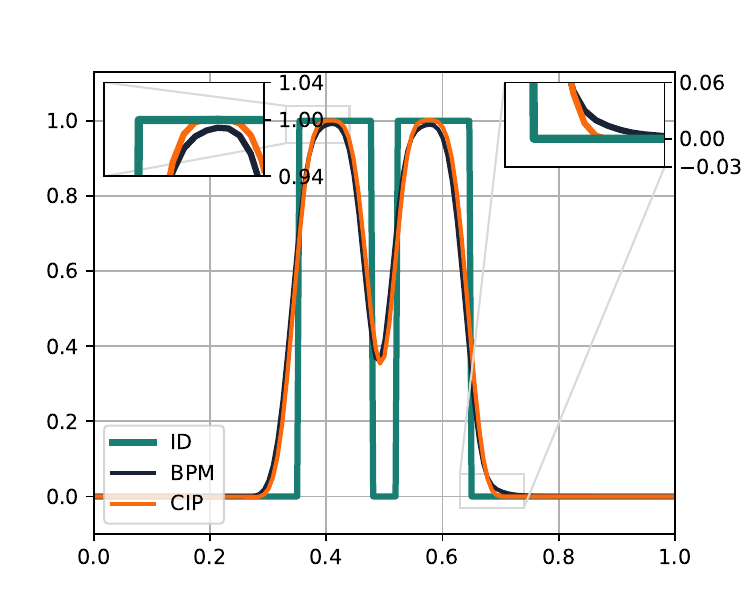}
	}
	\subfloat[CN, $\mathbb{Q}_1^{}$ elements]{
		\includegraphics[width=0.46\textwidth]{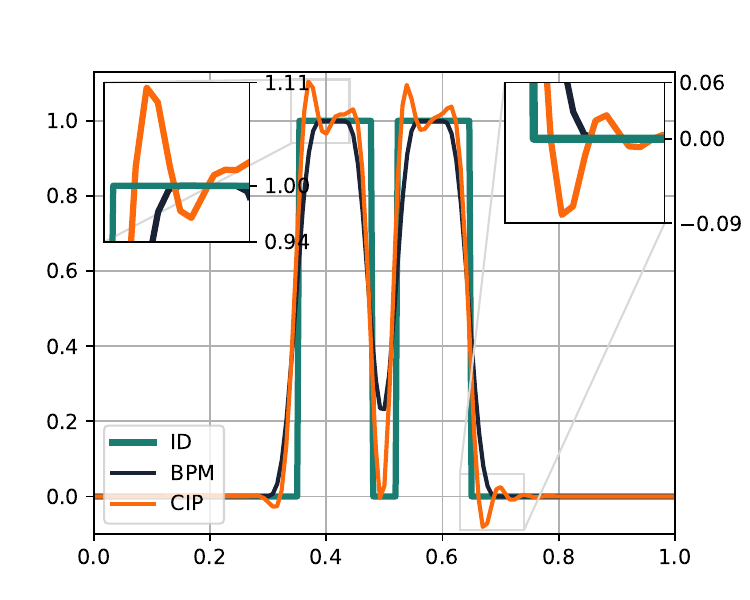}
	}
	\caption{  Cross sections were taken along the line $y=0.75$ of Initial data $u^{0}$, BP-Euler, BP-CN, CIP-Euler and CIP-CN methods at $T=6.28$ ($\gamma=0.001$, $P=130$). For plotting these cross-sections, when $\mathbb{P}_2^{}$ elements are used 10,000 equidistant points were chosen along the line $y=0.75$, and the values of the approximated solution have been plotted at these points.}\label{Cross123}
\end{figure}  
\begin{figure}[h!]
	\centering
	\subfloat[Euler, $\mathbb{P}_1^{}$ elements]{
		\includegraphics[width=0.5\textwidth]{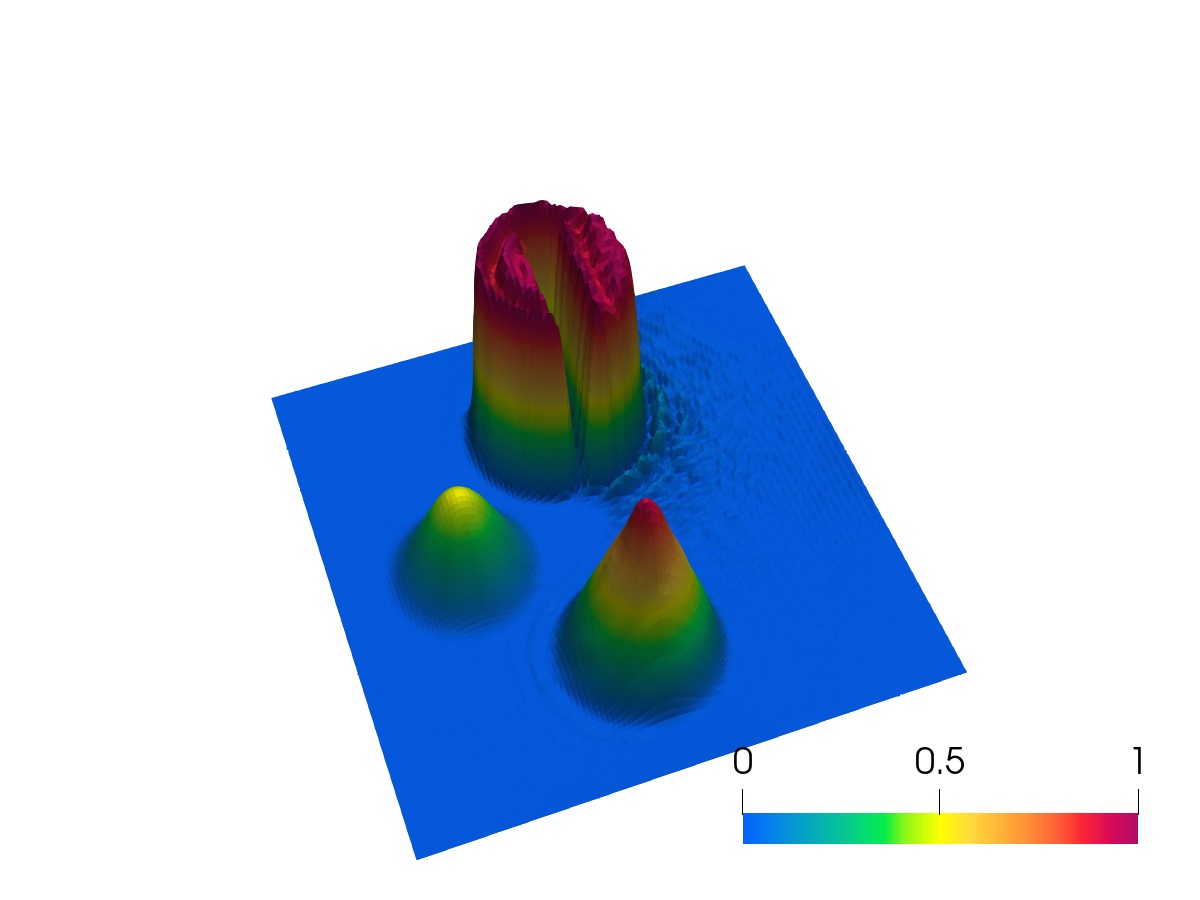}
	}
	\subfloat[CN, $\mathbb{P}_1^{}$ elements]{
		\includegraphics[width=0.5\textwidth]{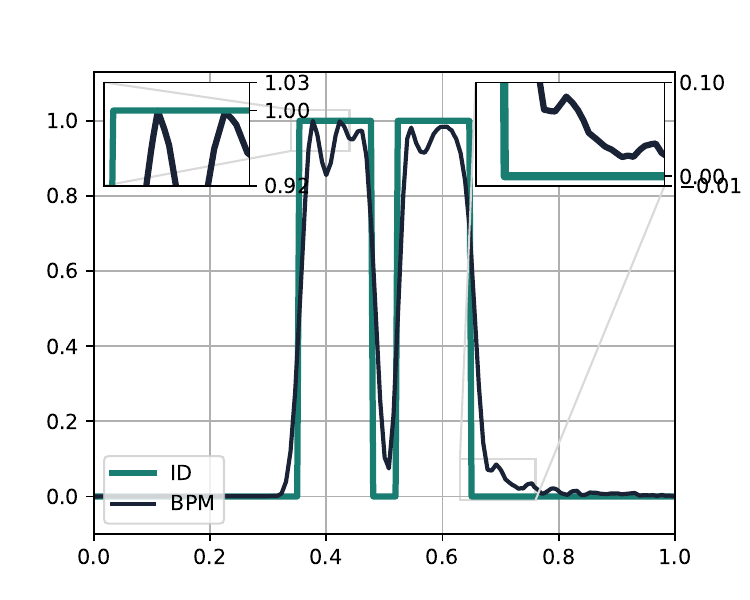}
	}
	\caption{ {\bf Left:} The approximation of the solution of Example 2 for BP-CN method without CIP term ($\gamma=0$) using $\mathbb{P}_1^{}$ elements and mesh \ref{e11}  ($P=130$) {\bf Right:} Cross sections were taken along the line $y=0.75$ of Initial data $u^{0}$ and BP-CN method without CIP term at $T=6.28$. }\label{FigurenonCIP}
\end{figure}  

\begin{figure}[h!]\label{Fig:Ex2-NonDel}
	\centering
	\subfloat[Euler, $\mathbb{P}_1^{}$ elements]{
		\includegraphics[width=0.46\textwidth]{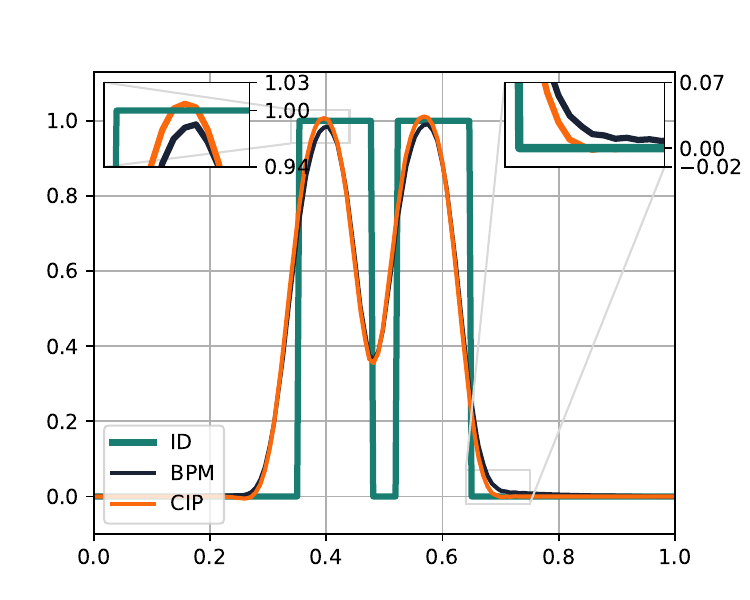}
	}
	\subfloat[CN, $\mathbb{P}_1^{}$ elements]{
		\includegraphics[width=0.46\textwidth]{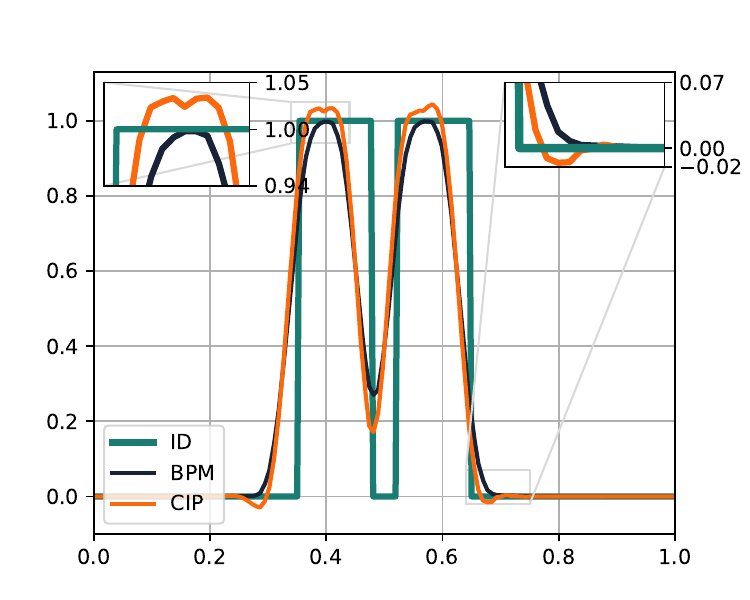}
	}
	\caption{  Cross sections were taken along the line $y=0.75$ of Initial data $u^{0}$, BP-Euler, BP-CN, CIP-Euler and CIP-CN methods at $T=6.28$ ($\gamma=0.001$, $P=130$) on the non-Delaunay mesh \ref{ND11}. }
\end{figure}  
\begin{figure}[h!]\label{Fig:MASS}
	\centering
	\subfloat[Euler]{
		\includegraphics[width=0.46\textwidth]{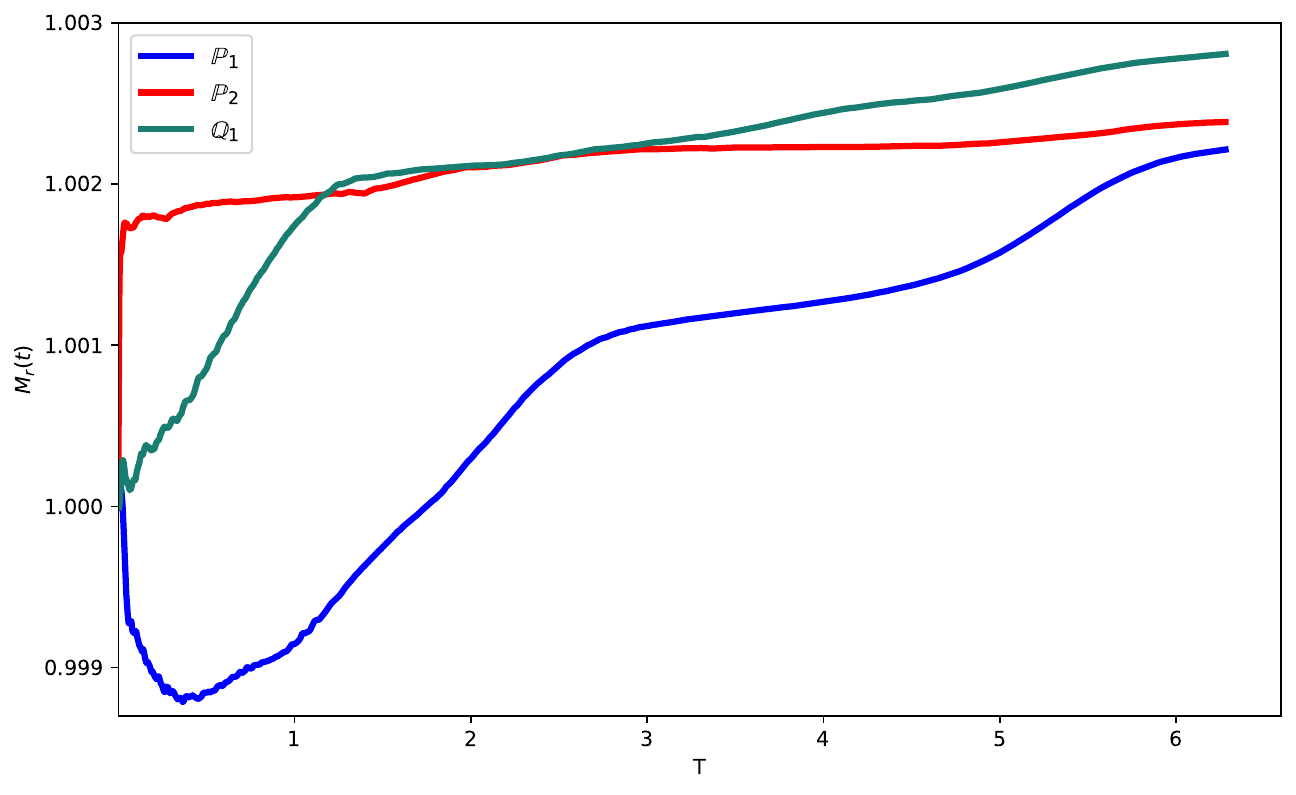}
	}
	\subfloat[CN]{
		\includegraphics[width=0.46\textwidth]{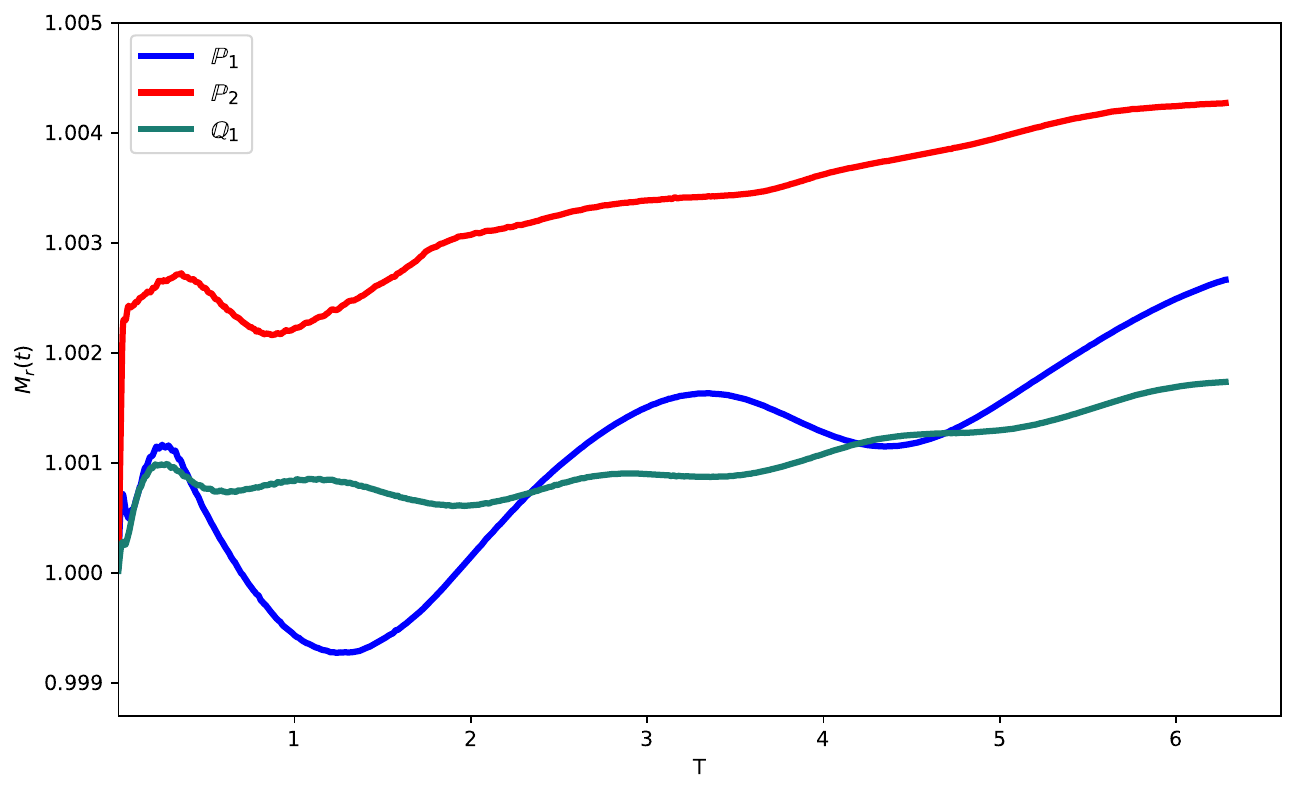}
	}
	\caption{  The evolution of mass over time employing the BP-Euler and BP-CN schemes. These methods have been implemented with $\mathbb{P}_1^{}$ and $\mathbb{P}_2^{}$ elements on Mesh~\ref{e11}, and $\mathbb{Q}_1^{}$ elements on Mesh~\ref{Q11}.
	 }
\end{figure}  

\section{Conclusions and outlook} 	
\label{sec:conc}
In this work we have build-up on the works \cite{BGPV23,ABT} and extended that framework to the time-dependent convection-diffusion equation.  The analysis of stability and error estimates
has been restricted to the implicit Euler scheme, but the method itself has also been tested in the context of the Crank-Nicolson time discretisation. For both options, the numerical experiments
show a good performance, with the solution respecting the physical bounds, while at the same time the layers in the solution do not seem to be excessively smeared.

It is important to mention that more economical alternatives, such as linearised flux-corrected transport (FCT) methods (see, e.g., \cite{MR2501695}), are also available to ensure bound preservation. Nevertheless, several factors should be considered. One aspect is the CFL condition, as most linear flux-corrected transport methods require such a condition to guarantee bound preservation, whereas our approach does not impose this restriction.  Another consideration is the applicability to higher-order elements. FCT methods have primarily been developed for linear finite elements, and bound preservation is not guaranteed for higher-order elements, as the necessary analysis has not yet been carried out. Additionally, improvements to the nonlinear solver itself can enhance performance. In this paper, we used a simple Richardson type solver to highlight the simplicity of the scheme, but more efficient nonlinear solvers, such as localised Newton methods (see, e.g.,  \cite{argyros1988newton}), or active set methods \cite{AAP25}, can significantly improve convergence speed. Our preliminary results indicate that these alternatives provide much faster convergence.
%
%	It is important to mention that more economical alternatives, such as linearised flux corrected transport methods (see, e.e., \cite{MR2501695}), are also available to generate bound preservation. Nevertheless, several factors should be considered, such as
%	
%	\begin{itemize}
%		\item[(i)] \textbf{CFL Condition:} Most linear FCT methods require a CFL condition to guarantee bound preservation. In contrast, our approach does not impose such a restriction.
%		
%		\item[(ii)] \textbf{Applicability to Higher-Order Elements:} FCT methods have primarily been developed for linear finite elements. Bound-preservation is not guaranteed for higher-order elements, and the analysis has not been carried out.
%		
%		\item[(iii)] \textbf{Nonlinear Solver Improvements:} The nonlinear solver itself can be improved. In this paper, we used a basic solver to highlight the simplicity of the scheme. More efficient nonlinear solvers, such as localised Newton methods (see, e.g., \cite{argyros1988newton}) or advanced iterative techniques (see, e.g., \cite{benson2006flexible,balay2022petsc}) can significantly enhance convergence speed. Our preliminary results indicate that these alternatives provide much faster convergence.  Alternatively,  the use of active set methods, such it has been done in \cite{AAP25} to tackle a variational
%inequality related to \eqref{variational} is currently underway.
%	\end{itemize}

Several problems
remain open at this point.  The extension of the stability and error analysis to higher-order time discretisation is, at the moment, an open problem. In addition,  the extension of this framework
to the transport equation is also of interest. A parallel development is the extension of this methodology to discontinuous Galerkin scheme in space, which is the topic of the companion paper \cite{barrenechea2024nodally}.  These, and other topics will be the subject of future research.

\paragraph{Acknowledgements}
The work of AA, GRB, and TP has been partially supported by the
Leverhulme Trust Research Project Grant No. RPG-2021-238.  TP is also
partially supported by EPRSC grants
\href{https://gow.epsrc.ukri.org/NGBOViewGrant.aspx?GrantRef=EP/W026899/2}
{EP/W026899/2},
\href{https://gow.epsrc.ukri.org/NGBOViewGrant.aspx?GrantRef=EP/X017206/1}
{EP/X017206/1}
and
\href{https://gow.epsrc.ukri.org/NGBOViewGrant.aspx?GrantRef=EP/X030067/1}
{EP/X030067/1}. 
%The authors also want to thank Emmanuil Geourgoulis
%and Andreas Veeser for many very helpful discussions.

%	\newpage

\bibliographystyle{elsarticle-num}
\bibliography{refs.bib}

\begin{thebibliography}{10}
\expandafter\ifx\csname url\endcsname\relax
  \def\url#1{\texttt{#1}}\fi
\expandafter\ifx\csname urlprefix\endcsname\relax\def\urlprefix{URL }\fi
\expandafter\ifx\csname href\endcsname\relax
  \def\href#1#2{#2} \def\path#1{#1}\fi

\bibitem{MR0679322}
A.~N. Brooks, T.~J.~R. Hughes,
  \href{https://doi.org/10.1016/0045-7825(82)90071-8}{Streamline
  upwind/{P}etrov-{G}alerkin formulations for convection dominated flows with
  particular emphasis on the incompressible {N}avier-{S}tokes equations},
  Comput. Methods Appl. Mech. Engrg. 32~(1-3) (1982) 199--259, fENOMECH ''81,
  Part I (Stuttgart, 1981).
\newblock \href {https://doi.org/10.1016/0045-7825(82)90071-8}
  {\path{doi:10.1016/0045-7825(82)90071-8}}.
\newline\urlprefix\url{https://doi.org/10.1016/0045-7825(82)90071-8}

\bibitem{HUGHES1989173}
T.~J. Hughes, L.~P. Franca, G.~M. Hulbert,
  \href{https://www.sciencedirect.com/science/article/pii/0045782589901114}{A
  new finite element formulation for computational fluid dynamics: Viii. the
  galerkin/least-squares method for advective-diffusive equations}, Computer
  Methods in Applied Mechanics and Engineering 73~(2) (1989) 173--189.
\newblock \href {https://doi.org/https://doi.org/10.1016/0045-7825(89)90111-4}
  {\path{doi:https://doi.org/10.1016/0045-7825(89)90111-4}}.
\newline\urlprefix\url{https://www.sciencedirect.com/science/article/pii/0045782589901114}

\bibitem{BURMAN20101114}
E.~Burman,
  \href{https://www.sciencedirect.com/science/article/pii/S0045782509003983}{Consistent
  {S}{U}{P}{G}--method for transient transport problems: Stability and
  convergence}, Computer Methods in Applied Mechanics and Engineering 199~(17)
  (2010) 1114--1123.
\newblock \href {https://doi.org/https://doi.org/10.1016/j.cma.2009.11.023}
  {\path{doi:https://doi.org/10.1016/j.cma.2009.11.023}}.
\newline\urlprefix\url{https://www.sciencedirect.com/science/article/pii/S0045782509003983}

\bibitem{Bochev04}
P.~B. Bochev, M.~D. Gunzburger, J.~N. Shadid,
  \href{https://www.sciencedirect.com/science/article/pii/S0045782504000830}{Stability
  of the {SUPG} finite element method for transient advection–diffusion
  problems}, Computer Methods in Applied Mechanics and Engineering 193~(23)
  (2004) 2301--2323.
\newblock \href {https://doi.org/https://doi.org/10.1016/j.cma.2004.01.026}
  {\path{doi:https://doi.org/10.1016/j.cma.2004.01.026}}.
\newline\urlprefix\url{https://www.sciencedirect.com/science/article/pii/S0045782504000830}

\bibitem{MR1736900}
J.-L. Guermond, \href{https://doi.org/10.1051/m2an:1999145}{Stabilization of
  {G}alerkin approximations of transport equations by subgrid modeling}, M2AN
  Math. Model. Numer. Anal. 33~(6) (1999) 1293--1316.
\newblock \href {https://doi.org/10.1051/m2an:1999145}
  {\path{doi:10.1051/m2an:1999145}}.
\newline\urlprefix\url{https://doi.org/10.1051/m2an:1999145}

\bibitem{MR2068903}
E.~Burman, P.~Hansbo, \href{https://doi.org/10.1016/j.cma.2003.12.032}{Edge
  stabilization for {G}alerkin approximations of convection-diffusion-reaction
  problems}, Comput. Methods Appl. Mech. Engrg. 193~(15-16) (2004) 1437--1453.
\newblock \href {https://doi.org/10.1016/j.cma.2003.12.032}
  {\path{doi:10.1016/j.cma.2003.12.032}}.
\newline\urlprefix\url{https://doi.org/10.1016/j.cma.2003.12.032}

\bibitem{MR2121360}
R.~Becker, M.~Braack, A two-level stabilization scheme for the
  {N}avier-{S}tokes equations, in: Numerical mathematics and advanced
  applications, Springer, Berlin, 2004, pp. 123--130.

\bibitem{BF09}
E.~Burman, M.~A. Fernández,
  \href{https://www.sciencedirect.com/science/article/pii/S0045782509000851}{Finite
  element methods with symmetric stabilization for the transient
  convection–diffusion–reaction equation}, Computer Methods in Applied
  Mechanics and Engineering 198~(33) (2009) 2508--2519.
\newblock \href {https://doi.org/https://doi.org/10.1016/j.cma.2009.02.011}
  {\path{doi:https://doi.org/10.1016/j.cma.2009.02.011}}.
\newline\urlprefix\url{https://www.sciencedirect.com/science/article/pii/S0045782509000851}

\bibitem{CR73}
P.~G. Ciarlet, P.-A. Raviart, Maximum principle and uniform convergence for the
  finite element method, Comput. Methods Appl. Mech. Engrg. 2 (1973) 17--31.

\bibitem{XZ99}
J.~Xu, L.~Zikatanov, \href{https://doi.org/10.1090/S0025-5718-99-01148-5}{A
  monotone finite element scheme for convection-diffusion equations}, Math.
  Comp. 68~(228) (1999) 1429--1446.
\newblock \href {https://doi.org/10.1090/S0025-5718-99-01148-5}
  {\path{doi:10.1090/S0025-5718-99-01148-5}}.
\newline\urlprefix\url{https://doi.org/10.1090/S0025-5718-99-01148-5}

\bibitem{BJK23}
G.~R. Barrenechea, V.~John, P.~Knobloch,
  \href{https://doi.org/10.1137/22M1488934}{Finite element methods respecting
  the discrete maximum principle for convection-diffusion equations}, SIAM Rev.
  66~(1) (2024) 3--88.
\newblock \href {https://doi.org/10.1137/22M1488934}
  {\path{doi:10.1137/22M1488934}}.
\newline\urlprefix\url{https://doi.org/10.1137/22M1488934}

\bibitem{MH85}
A.~Mizukami, T.~J.~R. Hughes, A {P}etrov-{G}alerkin finite element method for
  convection-dominated flows: an accurate upwinding technique for satisfying
  the maximum principle, Comput. Methods Appl. Mech. Engrg. 50~(2) (1985)
  181--193.
\newblock \href {https://doi.org/10.1016/0045-7825(85)90089-1}
  {\path{doi:10.1016/0045-7825(85)90089-1}}.

\bibitem{BE05}
E.~Burman, A.~Ern, Stabilized {G}alerkin approximation of
  convection-diffusion-reaction equations: discrete maximum principle and
  convergence, Math. Comp. 74~(252) (2005) 1637--1652 (electronic).
\newblock \href {https://doi.org/10.1090/S0025-5718-05-01761-8}
  {\path{doi:10.1090/S0025-5718-05-01761-8}}.

\bibitem{Kuz07}
D.~Kuzmin, Algebraic flux correction for finite element discretizations of
  coupled systems, in: M.~Papadrakakis, E.~O{\~n}ate, B.~Schrefler (Eds.),
  Proceedings of the Int.~Conf.~on Computational Methods for Coupled Problems
  in Science and Engineering, CIMNE, Barcelona, 2007, pp. 1--5.

\bibitem{BBK17-NumMath}
G.~R. Barrenechea, E.~Burman, F.~Karakatsani, Edge-based nonlinear diffusion
  for finite element approximations of convection-diffusion equations and its
  relation to algebraic flux-correction schemes, Numer. Math. 135~(2) (2017)
  521--545.

\bibitem{LHV13}
C.~Lu, W.~Huang, E.~S. {Van Vleck},
  \href{https://www.sciencedirect.com/science/article/pii/S0021999113001307}{The
  cutoff method for the numerical computation of nonnegative solutions of
  parabolic pdes with application to anisotropic diffusion and lubrication-type
  equations}, Journal of Computational Physics 242 (2013) 24--36.
\newblock \href {https://doi.org/https://doi.org/10.1016/j.jcp.2013.01.052}
  {\path{doi:https://doi.org/10.1016/j.jcp.2013.01.052}}.
\newline\urlprefix\url{https://www.sciencedirect.com/science/article/pii/S0021999113001307}

\bibitem{BGPV23}
G.~R. Barrenechea, E.~H. Georgoulis, T.~Pryer, A.~Veeser,
  \href{https://doi.org/10.1093/imanum/drad055}{A nodally bound-preserving
  finite element method}, IMA J. Numer. Anal. 44~(4) (2024) 2198--2219.
\newblock \href {https://doi.org/10.1093/imanum/drad055}
  {\path{doi:10.1093/imanum/drad055}}.
\newline\urlprefix\url{https://doi.org/10.1093/imanum/drad055}

\bibitem{ABT}
A.~Amiri, G.~R. Barrenechea, T.~Pryer,
  \href{https://doi.org/10.1142/S0218202524500283}{A nodally bound-preserving
  finite element method for reaction--convection--diffusion equations}, Math.
  Models Methods Appl. Sci. 34~(8) (2024) 1533--1565.
\newblock \href {https://doi.org/10.1142/S0218202524500283}
  {\path{doi:10.1142/S0218202524500283}}.
\newline\urlprefix\url{https://doi.org/10.1142/S0218202524500283}

\bibitem{EG21-I}
A.~Ern, J.-L. Guermond, Finite Elements I, Springer, 2021.

\bibitem{MR2759829}
H.~Brezis, Functional analysis, {S}obolev spaces and partial differential
  equations, Universitext, Springer, New York, 2011.

\bibitem{evans2010partial}
L.~C. Evans, Partial differential equations, 2nd Edition, Vol.~19 of Graduate
  Studies in Mathematics, American Mathematical Society, Providence, RI, 2010.

\bibitem{Makridakis:2018}
C.~G. Makridakis, \href{https://doi.org/10.1007/s00211-018-0955-5}{On the
  {B}abu\v ska-{O}sborn approach to finite element analysis: {$L^2$} estimates
  for unstructured meshes}, Numer. Math. 139~(4) (2018) 831--844.
\newblock \href {https://doi.org/10.1007/s00211-018-0955-5}
  {\path{doi:10.1007/s00211-018-0955-5}}.
\newline\urlprefix\url{https://doi.org/10.1007/s00211-018-0955-5}

\bibitem{EG21-II}
A.~Ern, J.-L. Guermond, Finite Elements II, Springer, 2021.

\bibitem{RST08}
H.-G. Ross, M.~Stynes, L.~Tobiska, Robust Numerical Methods for Singularly
  Perturbed Differential Equations, SSCM, volume 24, Springer, 2008.

\bibitem{renardy2006introduction}
M.~Renardy, R.~C. Rogers, An introduction to partial differential equations,
  2nd Edition, Vol.~13 of Texts in Applied Mathematics, Springer-Verlag, New
  York, 2004.

\bibitem{heywood1990finite}
J.~G. Heywood, R.~Rannacher, Finite-element approximation of the nonstationary
  navier--stokes problem. part iv: error analysis for second-order time
  discretization, SIAM Journal on Numerical Analysis 27~(2) (1990) 353--384.

\bibitem{MR1388492}
R.~J. Leveque, \href{https://doi.org/10.1137/0733033}{High-resolution
  conservative algorithms for advection in incompressible flow}, SIAM J. Numer.
  Anal. 33~(2) (1996) 627--665.
\newblock \href {https://doi.org/10.1137/0733033} {\path{doi:10.1137/0733033}}.
\newline\urlprefix\url{https://doi.org/10.1137/0733033}

\bibitem{MR2501695}
D.~Kuzmin, \href{https://doi.org/10.1016/j.jcp.2008.12.011}{Explicit and
  implicit {FEM}-{FCT} algorithms with flux linearization}, J. Comput. Phys.
  228~(7) (2009) 2517--2534.
\newblock \href {https://doi.org/10.1016/j.jcp.2008.12.011}
  {\path{doi:10.1016/j.jcp.2008.12.011}}.
\newline\urlprefix\url{https://doi.org/10.1016/j.jcp.2008.12.011}

\bibitem{argyros1988newton}
I.~K. Argyros, On newton's method and nondiscrete mathematical induction,
  Bulletin of the Australian Mathematical Society 38~(1) (1988) 131--140.

\bibitem{AAP25}
B.~S. Ashby, A.~Hamdan, T.~Pryer, \href{https://arxiv.org/abs/2501.11042}{A
  nodally bound-preserving finite element method for hyperbolic
  convection-reaction problems} (2025).
\newblock \href {http://arxiv.org/abs/2501.11042} {\path{arXiv:2501.11042}}.
\newline\urlprefix\url{https://arxiv.org/abs/2501.11042}

\bibitem{barrenechea2024nodally}
G.~R. Barrenechea, T.~Pryer, A.~Trenam, A nodally bound-preserving
  discontinuous {Galerkin} method for the drift-diffusion equation, arXiv
  preprint arXiv:2410.05040 (2024).

\end{thebibliography}



%	\appendix
%	
%	\section{Omitted Proof in Section~\ref{sec:examples}}
%	\label{app:1}
%	
%	\lipsum[7]

\end{document}